\documentclass[10pt]{amsart}
\pdfoutput=1
\usepackage{amsmath, amsthm, amssymb,slashed,stmaryrd}
\usepackage{ifpdf}
\usepackage[pdftex]{graphicx}
\usepackage{tikz-cd}
\usetikzlibrary{matrix,arrows,calc}
\usepackage{mathtools}
\usepackage{todonotes}
\usepackage[pdftex,plainpages=false,hypertexnames=false,pdfpagelabels]{hyperref}
\usepackage{float}
\usepackage{tikz}
\usepackage{dsfont}
\usetikzlibrary{matrix}

  \usepackage{url}
 \theoremstyle{plain}
 \newtheorem{thm}{Theorem}[section]
    \newtheorem{claim}{Claim}[section]
\newtheorem{defn/thm}[thm]{Definition/Theorem}
 \newtheorem{cor}[thm]{Corollary}
 \newtheorem{lem}[thm]{Lemma}
 \newtheorem{prop}[thm]{Proposition}
 \newtheorem{conj}[thm]{Conjecture}
 \theoremstyle{definition}
 \newtheorem{defn}[thm]{Definition}

 \newtheorem*{thm*}{Theorem}
 \theoremstyle{remark}
 \newtheorem{rmk}[thm]{Remark}
 
 \newtheorem{question}[thm]{Question}
 
\def\beq{\begin{eqnarray}}
\def\eeq{\end{eqnarray}}
 \newcommand{\bp}{\begin{proof}[Proof]}
 \newcommand{\ep}{\end{proof}}

\DeclareSymbolFont{bbold}{U}{bbold}{m}{n}
\DeclareSymbolFontAlphabet{\mathbbold}{bbold}

\def\Spec{{\rm Spec}}

\def\Gal{{\rm Gal}}
\def\Pic{{\rm Pic}}

\def\K3{{\rm K3}}
\def\Imag{{\rm Im}}

\def\GL{{\rm GL}}

\def\crys{{\rm crys}}

\def\Int{{\rm Int}}
\def\dR{{\rm dR}}
\def\Betti{{\rm Betti}}
\def\height{{\rm ht}}
\def\gr{{\rm gr}}
\def\Ext{{\rm Ext}}
\def\Tors{{\rm Tors}}
\def\Brauer{{\rm Br}}

\newcommand{\pdiv}{\mc{G}}
\newcommand{\defeq}{\vcentcolon=}

\newcommand\nc{\newcommand}
\nc\mf\mathfrak
\nc\mc\mathcal
\nc\mb\mathbb

\usepackage[
backend=biber,
style=alphabetic,
sorting=nyt
]{biblatex}

\begin{document}

\title{Calabi-Yau threefolds over finite fields and torsion in cohomologies}

\author{Yeuk Hay Joshua Lam}

\date{\today}

\begin{abstract} 
We study various examples of Calabi-Yau threefolds over finite fields. In particular, we provide a counterexample to a conjecture of K. Joshi on lifting Calabi-Yau threefolds to characteristic zero. We also compute the $p$-adic  cohomologies of some  Calabi-Yau threefolds constructed by Cynk-van Straten which have remarkable arithmetic  properties, as well as those of  the Hirokado threefold. These examples and computations answer some outstanding questions of B. Bhatt, T. Ekedahl,  van der Geer-Katsura and Patakfalvi-Zdanowicz, and   shed new light on the Beauville-Bogomolov decomposition in positive characteristic. Our tools include $p$-adic Hodge theory as well as classical algebraic topology. We also give potential examples showing that   Hodge numbers of threefolds in positive characteristic are not derived invariants, contrary to the case of characteristic zero.
\end{abstract}

\maketitle 
\tableofcontents
\setcounter{tocdepth}{1}
%\tableofcontents
\section{Introduction}
In this paper we investigate various  arithmetic properties of  Calabi-Yau varieties over finite fields. The common theme tying together all the results of this paper is that for a variety over a mixed characteristic local ring, the classical topology (including torsion) of the generic fiber has control over  all the cohomologies of the special fiber; the most general form of this phenomenon is due to the recent work of Bhatt-Morrow-Scholze, although we will only need older results of Caruso and Faltings.

\subsection{A conjecture of K. Joshi}
In a fascinating work \cite{joshi} K. Joshi studies many properties of surfaces and threefolds over finite fields, and states the following

\begin{conj}[Conjecture 7.7.1 of \cite{joshi}]\label{conj:joshiintro}
Let $X$ be a smooth proper Calabi-Yau threefold over a finite field $k$. Then $X$ lifts to characteristic zero if and only if 
\begin{enumerate}
    \item $H^0(X, \Omega^1)=0$, and 
    \item $X$ is \textit{classical}.
\end{enumerate}
\end{conj}

Here we may take as definition the adjective \textit{classical} as meaning (for a threefold) $H^3(X, \bar{\mb{Q}}_l)=0$ (it is a consequence of the results in   \cite{joshi} that this is equivalent to the original definition of classical, which is in terms of Hodge-Witt cohomology; for the original definition we refer the reader to Section 7.4 of loc.cit.).  Note that if a Calabi-Yau threefold  is liftable to characteristic zero then it is automatically classical, by Artin comparison and classical Hodge theory over $\mb{C}$.

\begin{thm}\label{thm:joshiintro}
There exists a  counterexample over $\mb{F}_5$ to Conjecture \ref{conj:joshiintro}. More precisely, there exists a liftable $X/\mb{F}_5$ with $H^0(X, \Omega^1)\neq 0$.
\end{thm}
The deformation theory of Calabi-Yau threefolds in positive characteristic is  mysterious, and most existing results require liftings to characteristic zero, making the latter an important problem to study.  As mentioned above, Hirokado constructed the first example of non-liftable ones, and by now there are several more examples of such; we refer the reader to \cite{takayamasurvey} for an excellent survey. %In particular, as far as the author knows, we do not know whether liftability to $W_2(k)$ implies liftability to $W_3(k)$.
In particular, the results of this paper show that all known examples of non-liftable Calabi-Yau threefolds satisfy $b_3=0$, i.e. the first condition in Joshi's conjecture, and it seems reasonable to ask if this is always the case:
\begin{question}
Do all non-liftable Calabi-Yau threefolds $X/\overline{\mb{F}}_p$ satisfy $b_3(X)=0$? 
\end{question} 

\subsection{Beauville-Bogomolov decomposition in positive characteristic} Over the complex numbers, varieties with trivial canonical bundle are built from three classes of varieties: abelian, hyperk"ahler, and Calabi-Yau. This important result is known as the Beauville-Bogomolov decomposition;  recently  Patakfalvi-Zdanowicz \cite{pz} have obtained interesting   results  towards a decomposition result in positive characteristic. It turns out that the Godeaux Calabi-Yau threefold example sheds new light on this question, and in particular answers a question in loc.cit. \cite[Question 13.6]{pz}. The details of this can be found in Section \ref{section:bbdecomposition}.

\subsection{The CvS Calabi-Yau threefolds}
We also study some beautiful Calabi-Yau threefolds constructed by Cynk-van Straten, who constructed a Calabi-Yau threefold $\mc{X}_5/\mc{O}$ (here $\mc{O}$ is the ring of integers of a certain number field) which is rigid in characteristic zero, but admits (obstructed) deformations in characteristic 5. Numerically what this means is that the Hodge number $h^{21}$ jumps from being $0$ to $1$ when we reduce $\bmod \ 5$; see Theorem \ref{thm:cvsintro} for the Hodge diamonds of both the generic and special fibers, and also \ref{thm:cyconditions} for more detail on these Calabi-Yau threefolds. We will refer to these as the CvS Calabi-Yau threefolds in what follows. 

We already saw an example of a Hodge number jumping in the special fiber in the counterexample to Joshi's conjecture above: there it was the Hodge number $h^{10}$ instead (and in fact on other components on the moduli space we may also have jumping behavior of $h^{01}$, or even both $h^{10}$ and $h^{01}$). There the reason was the torsion in the  first Betti cohomology of the generic fiber; in the case of the CvS Calabi-Yau threefold in characteristic $5$ we will show that the jumping behavior is not due to torsion in Betti cohomology, but rather a non-degeneration of the Hodge-de Rham spectral sequence. In fact, for the integral model $\mc{X}/\mc{O}$, all cohomologies are torsion free except for Hodge cohomologies, which are guaranteed to have torsion because of the jumping behaviour.

\begin{thm}\label{thm:cvsintro}
For the CvS Calabi-Yau threefold the following hold: 
\begin{enumerate}
    \item The variety $X_5/\mb{F}_5$ has torsion free crystalline cohomology; in particular the cohomology groups $H^*_{\dR}(\mc{X})$ of the integral model are torsion free as well. 
    
    \item The Hodge cohomology groups $H^i(\mc{X}, \Omega^j)$ are torsion free $\mc{O}$-modules, except for $H^2(\mc{X}, \Omega^1)$ and $H^2(\mc{X}, \Omega^2)$, which have  non-trivial  torsion.
    %\item Hodge symmetry holds for $X_5$: more precisely, the Hodge numbers are given by 
    %\[
    %h^{30}=1, \ h^{21}=1, \ h^{10}=h^{20}=h^{01}=h^{02}=0,
    %\]
    %and the rest are computed by Serre duality.
    \item The Hodge diamond of $X_5$ is given by 
\begin{equation*}
\begin{array}{ccccccc}
&&& 1 &&&\\
&& 0 && 0 &&\\
& 0 && 39 && 0 &\\ 
1 && 1 && 1 && 1\\
& 0 && 39 && 0 &\\ 
&& 0 && 0 &&\\
&&& 1 &&&.\\
\end{array}
\end{equation*}
This should be compared with that of the generic fiber $X/K$, which was computed by Cynk-van Straten \cite[Section 5.2]{cvs}:  
\begin{equation*}
\begin{array}{ccccccc}
&&& 1 &&&\\
&& 0 && 0 &&\\
& 0 && 38 && 0 &\\ 
1 && 0 && 0 && 1\\
& 0 && 38 && 0 &\\ 
&& 0 && 0 &&\\
&&& 1 &&&.\\
\end{array}
\end{equation*}
Note in particular that Hodge symmetry holds for $X_5$, which is not generally the case for smooth proper varieties over finite fields.
%\item The dimensions of the de Rham cohomology groups are given by 
%\[
%\dim H^*_{\dR}(X/k)= 1, \ 0,\  38,\  2,\  38,\  0,\  1.
%\]

    \item The special fiber $X_5$ is supersingular, i.e. its  Artin-Mazur formal group is isomorphic to $\hat{\mb{G}}_a$.
\end{enumerate}

\end{thm}

This theorem answers a question posed by  Bhatt \cite[p.48 2(c)]{bhattaws}, who asked for an example of a smooth  proper  variety over $\mc{O}_{\mc{C}_p}$ whose \'etale and crystalline cohomologies are torsion free but whose Hodge cohomologies are not (as well as an example where Hodge cohomology has more torsion than de Rham cohomology).
\begin{rmk}
Note that, for both the generic and special fibers, the Hodge cohomology groups $H^1(\Omega^2), H^1(\Omega^1)$ (as well as $H^1(\mc{O})$ which is part of the definition of a strict CY) were already computed by Cynk-van Straten in \cite{cvs}: our contribution for part (3) of the Theorem \ref{thm:cvsintro} is the computations of $H^0(\Omega^1), H^0(\Omega^2)$. On the other hand, it would also be  interesting to figure out the length of torsion in the Hodge cohomology groups $H^2(\mc{X}, \Omega^1)$ and  $H^2(\mc{X}, \Omega^2)$, since we have not been able to do so.
\end{rmk}

%Our tools in the computations leading to Theorem \ref{thm:cvsintro} include some $p$-adic Hodge theory, in particular the Theorem of Caruso and Faltings used in the counterexample to Joshi's conjecture, as well as some classical algebraic topology.

\subsection{Questions  of Ekedahl and  Takayama}
Our techniques also allow us to answer some questions of Ekedahl and Takayama. Since these are addressed using the same technique, we describe them in the same section. 

As mentioned above, the Hirokado threefold $\mc{H}$ was the first example of a non-liftable Calabi-Yau threefold. This example was studied in detail by Ekedahl \cite{ekedahl} by identifying $\mc{H}$ as a Deligne-Lusztig variety, from which he computed most cohomology groups  of $\mc{H}$, including all Hodge cohomologies (see Theorem \ref{thm:ekedahl}). However, the issue of the precise form of the crystalline cohomology was left open: in the computations of loc.cit. this boils down to the issue of  the degeneration of the Hodge-de Rham spectral sequence, which could  not be  resolved.

%On the other hand, in \cite{takayama} Takayama remarks that it is not known whether, for the Cynk-van Straten threefold $Y_3$, the group  $H^3_{\acute{e}t}(Y_3, \bar{\mb{Q}}_l)$ is zero, since this is a common feature of a non-liftable variety (as in the case of $\mc{H})$), and he asks whether .
On the other hand, Takayama asks \cite{takayamasurvey}[Section 4.1] whether the non-liftability of $Y_3$ can  again be explained by   the vanishing of $b_3$.

These two issues are resolved by the following 
\begin{cor}\label{cor:w2lift}
If a rigid Calabi-Yau threefold $X/k$ is not liftable to $W_2(k)$, then the Hodge-de Rham spectral sequence does not degenerate at the $E_1$-page, and $b_3=0$.  
\end{cor}
Indeed, Corollary \ref{cor:w2lift} implies that the Hodge-de Rham spectral sequence degenerates for $\mc{H}$, since $\mc{H}$ was shown to not lift to $W_2(k)$ by Ekedahl, thus resolving the first question above; on the other hand, Cynk-van Straten also showed that $Y_3$ does not lift to $W_2(k)$, and thus $b_3(Y_3)=0$, resolving the second issue stated above.
\subsection{Sketch of the arguments}
We comment briefly on the proofs of our results.

For the counterexample in Theorem \ref{thm:joshiintro}, $X$ is constructed by taking a free quotient of a quintic  hypersurface in $\mb{P}^4$ by a finite group of order $5$. Here we remark on why this is a plausible strategy: in fact we will consider the quotient $\mc{X}$ of a hypersurface $\mc{Y}$ over $\mb{Z}_5$ by a group scheme of order $5$, and let $X$ (respectively $Y$) denote the  special fiber of $\mc{X}$ (respectively $\mc{Y}$). In characterstic zero, the variety $\mc{X}[1/5]$ has $\pi_1=\mb{Z}/5$, and therefore by a result of Caruso %and Faltings (and also more recently Bhatt-Morrow-Scholze \cite{bms}),
we have   
\[
H^1_{dR}(X)\cong \mb{F}_5,
\]
where the left hand side is the algebraic de Rham cohomology of $X$. Thus, by the Hodge-de Rham spectral sequence, the groups $H^0(X, \Omega^1), \ H^1(X, \Omega^0)$ cannot both be zero. There are various choices of group schemes of order 5, and it turns out that choosing $\mu_5$ gives 
\[
H^0(X, \Omega^1)\neq 0, \ H^1(X, \mc{O})=0,
\]
which is what we wanted: the latter equality needed for the variety $X$ to be Calabi-Yau. For example, we can check that $H^1(X, \mc{O})=0$ as follows: since $X$ is a quotient of a hypersurface by $\mu_5$ we have  
\[
\Pic(X)\cong \mb{Z}/5,
\]
where this is an equality of group schemes, and hence 
\[
H^1(X, \mc{O})=T_e\Pic(Z)=0.
\]

%Now we will take the group to be $\mu_5$ in the special fiber.  We will also use  the following two  facts: first,  for  a free quotient by a finite group of the form  $X=Y/G$, the Picard group schemes  are related by
%\[
%1\rightarrow  G^D\rightarrow \Pic(X)\rightarrow \Pic(Y)\rightarrow 1,
%\]
%and second,  for any variety $Z$ we have 
%\[
%H^1(Z, \mc{O})=T_e\Pic(Z),
%\]
%where the right hand side denotes the tangent space at the identiy. Combining the two facts above we conclude   $H^1(X, \mc{O})=0$. Hence we must have $H^0(X, \Omega^1)$ non-zero, providing the desired counterexample. In fact, the Hodge-de Rham spectral sequence degenerated at the $E_1$-page, and in fact $H^0(X, \Omega^1)\cong \mb{F}_5$. For the computations we use a beautiful recent construction by M.Reid \cite{reid} of the \textit{Tate-Oort} group scheme $\mb{TO}_p$, which is a group scheme of order $p$ over a base ring $B\defeq \mb{Z}[S,t]/(St^{p-1}+p)$, such that its  reduction modulo $p$ has two components, one of which is $\mu_p$, the other is (a form of) $\mb{Z}/p$, and interecting at an $\alpha_p$.

\begin{rmk}
We could also have chosen to quotient by $\mb{Z}/5$ or $\alpha_5$. This example is reminiscent of the example of Enriques surfaces in characteristic 2, where the moduli space has two components according to whether $\Pic^{\tau} $ (the torsion part of the Picard group) is $\mu_2$ or  $\mb{Z}/2$, and intersecting along the locus where $\Pic^{\tau}\cong \alpha_2$.
\end{rmk}
\begin{rmk}
 As far as we are aware this  example was first considered by Aspinwall-Morrison \cite{aspinwallchiralrings} in their study of the \textit{discrete torsion} phenomenon in string theory, at least for the case of quotients of the Fermat hypersurface.
\end{rmk}

In order to  compute the  cohomologies required throughout this paper, we make frequent use of Caruso's theorem whenever we have characteristic zero lifts: the idea here is that Betti cohomologies of complex varieties are often easier to access, and using Caruso's result we can deduce information about the special fibers. This technique is perhaps most notable in the proof of Theorem \ref{thm:cvsintro}, and  seems to be the first application of integral $p$-adic Hodge theory to actual computations. To calculate Betti cohomlogies of the lifts we use classical tools such as  Mayer-Vietoris and the Serre spectral sequence.

\subsection{Some speculations}
We end in Section \ref{section:further} with some more possible applications of the main theme of this paper, namely the comparison between torsion in Betti cohomology of the generic fiber and cohomologies of the special fiber. 

We suggest possible examples showing that Hodge numbers are not derived invariants for threefolds in positive characteristic: this question was recently investigated in \cite{antieaubragg} and in characteristic zero it is a theorem of \cite{popaschnell} that the Hodge numbers of threefolds are in fact derived invariants. Our idea is that, since torsion in cohomology is not invariant (as shown by Addington \cite{addington}), and torsion contributes to Hodge cohomologies in the speical fiber, it seems likely that the Hodge cohomologies will difer in general. In \ref{section:hodgenumber} we give specific conjectural Calabi-Yau threefold examples; however, since the examples in question were ``constructed" physically by the use of  string dualities (cf. \cite{schulz}) and the author is not aware of rigorous mathematical constructions of these Calabi-Yau threefolds in the literature, we leave the details to future work.

In Section \ref{section:torsiondeformation} we consider the question of torsion in $H^3(X, \mb{Z})$ (the group $\Tors(H^3(X, \mb{Z}))$ is referred to as the Brauer group $\Brauer(X)$) for a Calabi-Yau threefold $X$. Again, a torsion class contributes to Hodge cohomologies of the reduction $\bmod \ p$, and if $X$ has a CY reduction then the contribution must be to $h^{21}$, which is also the dimension of the deformation space of $X$; in other words, classes in the Brauer group give extra $\bmod \ p$ directions for $X$  to deform in. 

We consider  some  particular examples of abelian surface fibered Calabi-Yau threefolds  with non-zero Brauer groups, and suggest how the extra deformations may be realized, analogous to the Moret-Bailly example \cite{moretbailly} of a family of mutually isogenous abelian surfaces in positive characteristic. 
%%%%%%%%%%%%%%%%%%%%%%%%%%
%\todo{Is it actually a super potential? Tells you GW invariants?}

\subsection{Acknowledgements} To be added.

\section{Godeaux Calabi-Yau threefolds}
In this section we introduce the Godeaux Calabi-Yau threefolds, using which we give counterexamples to Joshi's conjecture, as well as answer a question of van der Geer-Katsura. We also compute the Hodge diamond of these threefolds.

\subsection{Construction}
We make crucial use of some results of Kim-Reid on  the \textit{Tate-Oort group scheme} whose definition we now recall. Throughout this section we fix a prime $p$. For a group scheme $G$ over a base scheme $S$ and an element $f\in \mc{O}(S)$, we denote by $G[1/f]$ the group scheme over $S[1/f]$.

\begin{defn/thm}[Section 3 of \cite{reid}]\label{defn/thm:to}
Let $B$ denote the ring $\mb{Z}[S,t]/(P)$ where $P\defeq St^{p-1}+p$. There is an order $p$ group scheme $\mb{TO}_p$ such that 
\begin{itemize}
    \item $\mb{TO}_p[1/t]$ is isomorphic to the multiplicative group scheme $\mu_p$;
    \item $\mb{TO}_p[1/S] $ is a form of the \'etale group scheme $\mb{Z}/p$;
    \item the fiber of $\mb{TO}_p$ over the point $\Spec (\mb{F}_p)=\Spec(B/(S,t))$ is isomorphic to $\alpha_p$.
\end{itemize}
Here as usual, over a base scheme in characteristic $p$,  $\alpha_p$ denotes the order $p$ group scheme which is the kernel of Frobenius on $\mb{G}_a$. We refer to $\mb{TO}_p$ as the Tate-Oort group scheme.
\end{defn/thm}

There is the following simple description of $\mb{TO}_p$ over $\bar{B}\defeq B\otimes_{\mb{Z}}\mb{F}_p\cong  \mb{F}_p[S,t]/(St)$. We make the following definitions:
\begin{align}
&\mb{G}_{\bar{B}}\defeq \begin{pmatrix} 1 & 0 \\
                                      x & 1+tx
                                      \end{pmatrix} \subset \GL(2, \bar{B}),\\
&\overline{\mb{TO}}_p\defeq V(x^p=Sx)\subset \mb{G}_{\bar{B}} \label{eqn:top};
\end{align}
i.e. the second equation means we take the closed  subscheme of $\mb{G}_{\bar{B}}$ defined by the equation  $x^p=Sx$. One has to check  first of all that  $\mb{G}_{\bar{B}}$ is a group, and secondly that (\ref{eqn:top}) defines a subgroup: we refer the reader to  \cite[Proposition 3.1]{reid} for the details. Note that when $S$ is invertible the (\ref{eqn:top}) indeed defines an \'etale covering of $\Spec \bar{B}$; when $t$ invertible we may define  the map 
\begin{align}
    &\mb{G}_m\rightarrow \mb{G}_{\bar{B}},\\
    &\lambda \mapsto \begin{pmatrix}
                    1 & 0\\
                    \frac{\lambda -1}{t} & \lambda
                    \end{pmatrix}
\end{align} 
and observe that $\mb{TO}_p$ is precisely the image of $\mu_p\subset \mb{G}_m$. Finally when $t=S=0$ we have that $\mb{TO}_p$ is by definition the subgroup of 
\[
\mb{G}_a\cong \begin{pmatrix} 1 & 0 \\
                                      x & 1
                                      \end{pmatrix} \subset \GL(2, \bar{B}),
                                      \]
                                      satisfying $x^p=0$, which is also the definition of $\alpha_p$. To summarize this discussion, we  have written down an explicit group scheme $\bar{\mb{TO}}_p$ satisfying the three conditions specified by Definition/Theorem \ref{defn/thm:to} $\bmod \ p$.

We record here  the following theorem of Kim-Reid (see \cite[Section 6]{reid}, and also \cite{reidwebsite}).
\begin{thm}
There exists a quintic hypersurface $Y_5\subset \mb{P}^4_B$ with an action of $\mb{TO}_5$ such that the quotient $X_5$ is non-singular over the locus $S=t=0$. 
\end{thm}
%\begin{rmk}
%Note that this only gives a counterexample for one direction of the conjecture. It would be very interesting to investigate the opposite direction, namely the existence of a lifting given conditions (1) and (2).
%\end{rmk}
\subsection{Joshi's conjecture}
In \cite{joshi} Joshi studies many interesting questions about $p$-adic cohomologies of varieties, and states the following conjectural criterion for  lifting Calabi-Yau threefolds:

\begin{conj}[Conjecture 7.7.1 of \cite{joshi}]\label{conj:joshi}
Let $X$ be a smooth proper Calabi-Yau threefold over $k$. Then $X$ lifts to characteristic zero if and only if 
\begin{enumerate}
    \item $H^0(X, \Omega^1)=0$, and 
    \item $X$ is \textit{classical}.
\end{enumerate}
\end{conj}
In the statement of Conjecture \ref{conj:joshi} the adjective classical means one of several equivalent things, one of which is that the Betti number $b_3$ (the dimension of $H^3_{\acute{e}t}(X, \mb{Q}_l)$ for any $l\neq p$) is non-vanishing; we will take this as our definition in what follows. Note that if $X$ admits a lift  to characteristic zero $\tilde{X}$  then it is automatically classical: indeed, by Artin's comparison theorem 
\[
H^3_{\acute{e}t}(X, \bar{\mb{Q}}_l)\cong H^3(\tilde{X}(\mb{C}), \mb{Q})\otimes \bar{\mb{Q}}_l,
\]
where the right hand side denotes Betti cohomology of the topological space $\tilde{X}(\mb{C})$, and is non-zero by Hodge theory.

\begin{thm}\label{thm:joshi}
There exists a Calabi-Yau threefold $X/k$ with a lift to to characteristic zero and such that 
\[
H^0(X, \Omega^1)\cong k.
\]

\end{thm}
\begin{rmk}
Note that this only gives a counterexample for one direction of the conjecture. It would be very interesting to investigate the opposite direction, namely the existence of a lifting given conditions (1) and (2).
\end{rmk}

\begin{proof}[Proof of Theorem \ref{thm:joshi}]
Over the ring $B[1/S]$, the group scheme $\mb{TO}_5$ becomes the \'etale group scheme $\mb{Z}/5$, and by the non-singularity of the fiber of $X_5$ over $S=t=0$, we have that for a generic invertible $S$, the fiber of $X_5$ is also non-singular. Pick such a generic choice of $S$ and let  $\tilde{X}=X$ be the  corresponding fiber of $X_5$; also denote by  $X$ the  special fiber of $\tilde{X}$. Note that $X$ corresponds to some smooth point $x\in \Spec(B)(k)$ and we may assume $\tilde{X}$ corresponds to a $W(k)$-point of $\Spec(B)$.

Note that in characteristic zero we have 
\[
H_1(\tilde{X}(\mb{C}), \mb{Z})=\mb{Z}/5,
\]
since $\tilde{X}(\mb{C})$  arises as a free quotient of a hypersurface by $\mb{Z}/5$, and therefore by the universal coefficients theorem 
\[
H^1(\tilde{X}(\mb{C}), \mb{Z}/5)=\mb{Z}/5.
\]
Also by Deligne-Illusie's theorem the Hodge-de Rham spectral sequence degenerates for the special fiber $X$: indeed,  $X$ lifts to $W(k)$  $\dim X=3<4=5-1$. Also, by Theorem \ref{thm:carusofaltings} we have 
\[
\dim_kH^1_{\dR}(X_k)=\dim_{\mb{F}_5}H^1(\tilde{X}(\mb{C}), \mb{Z}/5)=1
\]
%\[
%H^1_{dR}(X,k)\cong H^1(\tilde{X}(\mb{C}), \mb{F}_p)\otimes k,
%\]
and therefore $h^{10}(X)+h^{01}(X)=1$. Now since $h^{01}(X)=\dim T_e(\Pic(X))=0$, we conclude that $h^{10}\neq 0$, and hence 
\[
H^0(X, \Omega^1)\cong k,
\]
as required.
\end{proof}
\begin{rmk}
It is also possible to explicitly construct a non-vanishing section in $H^0(X, \Omega^1)$, as follows. Note that  we have the covering
\[
Y\rightarrow X,
\]
which is a $\mu_5$-cover by construction, and hence inseparable. By the same argument as in \cite[Proposition 0.1.2]{cossecdolgachev}, we may produce a section in $H^0(X, \Omega^1)$, whose vanishing locus is precisely the image of the singular locus of $Y$. We review the construction here: we have that 
\[
\pi_*\mc{O}_Y\cong \mc{O}_X\oplus \bigoplus_{i=1}^4\mc{L}^{\otimes i}
\]
for some line bundle $\mc{L}$ on $X$. Now since the map $\pi$ is purely inseparable, $Y$ is locally given by the equation
\[
y^5=b
\]
for some $b\in \Gamma(\mc{L}^{\otimes 5})$. Now $\mc{L}^{\otimes 5}\cong \mc{O}$, and one may check that $db \in \Gamma(\Omega^1_X)$ glue together to give a global section. See also  the discussion in \cite[Section 6.2.2]{reid} for the construction of this section.
\end{rmk}
\subsection{A question of van der Geer-Katsura}
In their study of heights of Calabi-Yau threefolds, van der Geer-Katsura asks the following \cite[Section 7]{vandergeerkatsura}
\begin{question}[van der Geer-Katsura]
Can a Calabi-Yau variety of dimension 3 in positive characteristic have non-zero regular 1-forms or regular 2-forms?
\end{question}
We observe that Theorem \ref{thm:joshi} gives a positive answer to the first part of this question. It would be interesting to figure out the answer to the second question, namely whether we can have $H^0(X, \Omega^2)\neq 0$. Note that even though the Godeaux Calabi-Yau threefold has an extra class in $H^2_{\dR}$, the class lands in $H^1(\Omega^1)$ by the following 
\begin{prop}\label{prop:global2forms}
For the Godeaux Calabi-Yau threefolds, we have $H^0(X, \Omega^2)=0$.
\end{prop}
\begin{proof}
Suppose not, i.e that $H^0(X, \Omega^2)\neq 0$. Note that the Godeaux surface  $S$ appears as a hyperplane section of $X$, and it is classical in the sense that, just like $X$, we have 
\[
\Pic^{\tau}(S)\cong \mb{Z}/5.
\]
By standard facts about classical surfaces \cite{lang}, we have $h^{20}(S)=0$, and therefore 
\[
H^0(S, \Omega^2_S)=0
\]
by Serre duality. Also (by the construction of classical Godeaux surfaces in \cite{lang}, as well as the construction of Godeaux Calabi-Yau threefolds in  Appendix \ref{appendix},  both $X$ and $S$ lift to characteristic zero: let us denote these lifts by $\tilde{X}$, $\tilde{S}$ respectively. Now  the  Lefschetz hyperplane theorem implies that the induced map on  Betti cohomology groups
\[
H^2(\tilde{X}, \mb{Z})\rightarrow H^2(\tilde{S}, \mb{Z})
\]
is injective. Then  by Caruso \cite[Th\'eor\`eme 1.1]{caruso} the map
\[
H^2_{\dR}(X)\rightarrow H^2_{\dR}(S)
\]
is injective also, since the \'etale cohomologies are obtained by applying Breuil's functor $T_{st\star}$. This is a contradiction since $H^0(S, \Omega^2_S)=0$, and therefore we have $H^0(X, \Omega^2_X)=0$, as claimed.
\end{proof}

\subsection{Beauville-Bogomolov decomposition}\label{section:bbdecomposition}

In an interesting recent paper\cite{pz}, Patakfalvi-Zdanowicz considered the problem of extending the Beauville-Bogomolov decomposition, which is a powerful tool in studying $K$-trivial varieties over the complex numbers,  to positive characteristic. Along the way, they posed the following Question, which we will answer in the present work.
\begin{question}[Question 13.6 \cite{pz}]\label{question:pz}
Let $k$ be an algebraically closed field of characteristic $p>0$. Does there exist a finite $\mu_p$-quotient $Y\rightarrow X$ such that $Y$ is a singular projective Gorenstein $K$-trivial variety and $X$ is a smooth, weakly ordinary, projective $K$-trivial variety over $k$ with $\hat{q}(X)=0$?
\end{question}
Here $\hat{q}$ denotes the \textit{augmented irregularity}, namely
\[
\hat{q}(X)\defeq \rm{max}\{ \dim \rm{Alb}_{X'}|X'\rightarrow X \ \text{is finite \'etale}\},
\]
where for a variety $X'$,  $\rm{Alb}_{X'}$ denotes its Albanese variety. 
\begin{rmk}
The interest in the question above is that a positive answer allows one to construct many $K$-trivial varieties which, for example,  have Calabi-Yau and abelian varieties parts which are ``difficult'' to separate. For example, one may take a quotient of the form $(Y\times E)/\mu_p$ with $E$ being an ordinary elliptic curve, and $\mu_p$ acts on $Y$ as in Question~\ref{question:pz}, and $\mu_p$ acts on $E$ by translation via the inclusion $\mu_p\hookrightarrow E[p]$. For details see \cite{pz}.
\end{rmk}
We answer Question \ref{question:pz} in the positive:
\begin{thm}
The pair  $(Y, X)$, with $Y$ a generic choice of hypersurface $X$ in Appendix \ref{appendix}, $X\defeq Y/\mu_5$ the smooth quotient, satisfies the conditions in Question \ref{question:pz}. 
\end{thm}

\begin{proof}
By Theorem \ref{thm:godeauxconstruct} the threefold $X$ is Calabi-Yau, so in particular $K$-trivial.  Note that the $\mu_5$-action is free for a generic choice of $Y$: that is, $Y$ has a nowhere vanishing vector field. Therefore $Y$ is singular, since otherwise the global vector field implies it  has vanishing Euler characteristic, which is not true since  a quintic hypersurface in $\mb{P}^4$ has Euler characteristic $-200$. On the other hand, $Y$ is certainly Gorenstein, being a hypersurface in $\mb{P}^4$. 

We now come to the augmented irregularity $\hat{q}(X)$. Note that  $X$ is (geometrically) simply connected, since it is homeomorphic to $Y$, which in turn is simply connected, since it is a hypersurface in $\mb{P}^4$. On the other hand, we have $H^1(X, \mc{O})=0$, and hence $\hat{q}(X)=0$, as required by Question \ref{question:pz}. Alternatively, for any $\ell$ we have 
\[
H^1_{\acute{e}t}(X, \mb{Q}_{\ell})=0
\]
since $X$ has a characteristic zero lift with fundamental group $\mb{Z}/5$, and therefore $\rm{Alb}_X$ is trivial. 

It remains to show that $X$ is weakly ordinary; in other words, it suffices to show that for a generic choice of  $X$, the action of Frobenius on $H^3(X, \mc{O})$ is invertible. Now  note  that the groups $H^3(Y, \mc{O})$, $H^3(X, \mc{O})$ are both one dimensional, and the pullback map 
\begin{equation}\label{eqn:pullback}
H^3(X,\mc{O}_X)\rightarrow H^3(Y, \mc{O}_Y)
\end{equation}
is an isomorphism: indeed, the map $\pi$ is finite and therefore $\pi_*$ is exact, and the above map is given by 
\[
H^3(X, \mc{O}_X)\rightarrow H^3(X, \pi_*\mc{O}_Y),
\]
and since 
\[
\pi_*\mc{O}_Y \cong \mc{O}_X\oplus \bigoplus_{i=1}^4\mc{L}^{\otimes i}
\]
for some line bundle $\mc{L}$ on $X$ (in fact, $\mc{L}$ corresponds to a  non-zero element of $\Pic(X)^{\tau}$, which gives rise to the covering $Y$), we have that the map (\ref{eqn:pullback}) is non-zero, and since both groups are one dimensional $k$-vector spaces, we have it being an isomorphism, as claimed. 

Now it suffices to check that Frobenius acts non-trivially  on the group $H^3(Y, \mc{O})$.

We recall that we have the $\mu_5$-action on $\mb{P}^4$ where $\zeta \in \mu_5$ acts via the map  
\[
(X_0: X_1:X_2:X_3:X_4)\mapsto  (X_0: \zeta X_1:\zeta^2 X_2:\zeta^3X_3:\zeta^4 X_4).
\]
The hypersurface $Y$ is defined by a quintic polynomial in the $X_i$'s which is invariant under this $\mu_5$-action. It suffices to exhibit a single such invariant quintic polynomial such that the Frobenius action on the corresponding group $H^3(Y, \mc{O})$ is non-trivial, since the condition of being weakly ordinary is an open one. There is the following simple formula for this action given by \cite{stienstra}: see Theorem \ref{thm:stienstra} below as well as Lemma \ref{lemma:frobenius}, and the fact that the tangent space to the formal group in Theorem \ref{thm:stienstra} is precisely $H^3(\mc{O})$. But now we can simply take the invariant quintic $X_0X_1X_2X_3X_4$, and it is trivial to check, using Theorem \ref{thm:stienstra}, that the Frobenius  action on $H^3(\mc{O})$ is non-trivial: indeed, by Theorem \ref{thm:stienstra}, the action of Frobenius is given by the coefficient $X_0^4X_1^4X_2^4X_3^4X_4^4$ in $(X_0X_1X_2X_3X_4)^4$, which is obviously non-zero, as required.

\end{proof}

\begin{thm}[{\cite[Theorem 1]{stienstra}}]\label{thm:stienstra}
For a hypersurface $Y\subset \mb{P}^N$ of degree $N+1$  defined by the equation $F(X_0, \cdots, X_N)=0$, there is a formal group law for $H^{N-1}(Y, \hat{\mb{G}}_{m,Y})$ whose logarithm $l(\tau)$ is given by
\[
l(\tau)=\sum_{m\geq 1}\beta_m\frac{\tau^m}{m},
\]
where 
\[
\beta_m=\text{coefficient of}  \ X_0^{m-1}\cdot \cdots \cdot X_N^{m-1} \ \text{in} \ F^{m-1}.
\]

\end{thm}
The following lemma relates the $p$th coefficient of the logarithm of a formal group to the action of Frobenius on its tangent space, which is certainly well known but for which we have not been able to find an adequate reference.
\begin{lem}\label{lemma:frobenius}
Suppose we have a  discrete valuation ring $R$ with maximal ideal $\mf{p}$, and   a one-dimensional formal group $\pdiv$ over $R$  with logarithm 
\[
l(\tau)=\sum_{m\geq 1} \beta_m\frac{\tau^m}{m};
\]
the action  of Frobenius on the tangent space of $\pdiv\otimes R/\mf{p}$ is given by multiplication by $\beta_p \bmod \ \mf{p}$.
\end{lem}
\begin{proof}
Let $F$ denote the Frobenius operator. Consider the Cartier module 
\[
\mc{C}\pdiv\defeq \lim \pdiv(tR[t]/t^n),
\]
which may be identified with $tR\llbracket t\rrbracket$.

Furthermore on $\mc{C} \pdiv$ We have the Frobenius operator $F$ (we only consider the Frobenius  at $p$) acting  by  the formula 
\[
F(\gamma(t))=\gamma(\zeta t^{1/p})\boxplus \gamma(\zeta^2t^{1/p})\boxplus \cdots \boxplus \gamma(\zeta^p t^{1/p}),  
\]
where $\zeta$ denotes a primitive $p$th root of unity,  the symbol $\boxplus$ denotes the addition in the formal group $\pdiv$, and $\gamma(t)$ denotes any element of $tR\llbracket t\rrbracket$ which we have identified with $\mc{C}\pdiv$ (c.f.\cite[p.1117]{stienstra}).

A straightforward  calculation  then shows that 
\begin{equation}\label{eqn:frobtangent}
l(F\tau)=\sum_{n\geq 1}\beta_{np}\frac{\tau^n}{n}.
\end{equation}
Since the map $l$ satisfies
\[
l(\tau)\cong \tau \bmod \ \tau^2,
\]
differentiating the  logarithm gives an isomorphism
\[
dl: T_e(\pdiv[1/p]) \cong T_e \mb{G}_{a, R[1/p]},
\]
(where for a group $\mc{H}$, we denote by $T_e\mc{H}$ the tangent space at the identity, and $\pdiv[1/p]$ denotes the base change of $\pdiv$ to the generic fiber of $\Spec (R)$) and the calculation (\ref{eqn:frobtangent}) shows that $F$ acts as $\beta_p$ on $T_e \mb{G}_{a, R[1/p]}$. Since the logarithm is Frobenius equivariant, $F$ also acts by $\beta_p$ on $T_e\pdiv[1/p]$. Therefore $F$ acts by $\beta_p \bmod \mf{p}$ on the $T_e(\pdiv\otimes R/\mf{p})$, as required. \end{proof}

\subsection{Hodge diamond}
We may now compute all the Hodge numbers of the Godeaux Calabi-Yau threefolds which are $\mu_5$-quotients of quintic hypersurfaces; the computation is straightforward in characteristic zero while somewhat tricky in characteristic 5.

In the following we  denote by $X$ the integral model of the Godeaux Calabi-Yau threefold, and $\overline{X}$ the special fiber. We now  compute the Hodge diamond of $\overline{X}$; first we  prove  the following topological result.
\begin{prop}\label{prop:h3torsion}
Let $X(\mb{C})$ denote the complex manifold associated to the generic fiber of $X$. %If $\tilde{X}$ is a characteristic zero lift of the Godeaux Calabi-Yau threefold $X$,
Then the  Betti cohomology group  $H^3(X(\mb{C}), \mb{Z})$ is torsion free.
\end{prop}
\begin{proof}
We apply the homology Serre spectral sequence to the fibration $X(\mb{C})\rightarrow K(\pi_1(X(\mb{C})))$, whose fiber is the universal cover of $X(\mb{C})$, which we denote by $\tilde{X}(\mb{C})$. We remind the reader that for a fibration
\[
F\rightarrow X\rightarrow B
\]
this  $E^2$-page  spectral sequence is given by 
\[
E^2_{p,q}=H_p(B, H_q(F)) \Rightarrow H_{p+q}(X),
\]
where all cohomologies are taken with integral coefficients. In the following we will denote by $d^i$ the differentials on the $i$th page.  In our setup the $E^2$-page    simplifies to
\begin{equation}\label{eqn:explicitserre}
E^2_{p,q}=H_p(\mb{Z}/5, H_q(\tilde{X}(\mb{C}))),
\end{equation}
where the right hand side of (\ref{eqn:explicitserre})   denotes group homology. Since $\tilde{X}(\mb{C})$ is simply connected, $H_1(\tilde{X}(\mb{C}))=0$, and therefore we have 
\[
E^2_{p,1}=0 \ \text{for all}\ p.
\]
Hence the differential 
\[
d^2: E^2_{2,1}\rightarrow E^2_{0,2}
\]
is zero, and we have  
\begin{align}\label{eqn:serress}
E^{\infty}_{0,2}&=E^2_{0,2}/\Imag(d^3)\\
                &=H_0(\mb{Z}/5, H_2(\tilde{X}(\mb{C})))/d^3(H_3(\mb{Z}/5, \mb{Z})).
\end{align}
Now recall that $\tilde{X}(\mb{C})$ is a hypersurface in $\mb{P}^4$, and therefore 
\[
H_2(\tilde{X}(\mb{C}))\cong \mb{Z}
\]
by the Lefschetz hyperplane theorem.  The action of $\mb{Z}/5$ on $H_2(\tilde{X}(\mb{C}))$  can only be  the trivial one, and therefore   we have 
\[
H_0(\mb{Z}/5, H_2(\tilde{X}(\mb{C})))\cong \mb{Z}.
\]
On the other hand $H_3(\mb{Z}/5, \mb{Z})\cong \mb{Z}/5$ which is torsion, and hence  $d^3(H_3(\mb{Z}/5, \mb{Z}))$  vanishes. Using (\ref{eqn:serress}) we conclude that 
\[
E^{\infty}_{0,2}\cong H_2(\tilde{X}(\mb{C}))\cong \mb{Z}.
\]
The filtration on $H_*(X(\mb{C}),\mb{Z})$ given by this spectral sequence gives a short exact sequence
\[
0\rightarrow E^{\infty}_{0,2}\rightarrow H_2(X(\mb{C}),\mb{Z})\rightarrow E^{\infty}_{2,0}\rightarrow 0.
\]
However, we also have 
\[
E^{\infty}_{2,0}\cong E^2_{2,0}\cong H_2(\mb{Z}/5, \mb{Z})=0,
\]
and therefore 
\[
H_2(\tilde{X}, \mb{Z})\cong \mb{Z}.
\]
We conclude that $H^3(X(\mb{C}), \mb{Z})$ is torsion free by the universal coefficients theorem, as claimed.
\end{proof}

This gives immediately the following
\begin{cor}
The Hodge diamond of $\overline{X}$ is  given by 
\begin{equation}\label{eqn:godeauxhodgediamond}
\begin{array}{ccccccc}
&&& 1 &&&\\
&& 1 && 0 &&\\
& 0 && 2 && 0 &\\ 
1 && 21 && 21 && 1\\
& 0 && 2 && 0 &\\ 
&& 0 && 1 &&\\
&&& 1 &&&.\\
\end{array}
\end{equation}
\end{cor}
\begin{proof}
By  Proposition \ref{prop:h3torsion} and Caruso's theorem, 
\[
H^2_{\dR}(\overline{X})\cong k^2, \ H^1_{\dR}(\overline{X})\cong k.
\]
Recall also that we have 
\[
H^0(\overline{X}, \Omega^1)\cong k, \ H^1(\overline{X}, \mc{O})\cong 0.
\]
Therefore (recall that $X$ denotes an integral lift of $\overline{X}$) $H^1(X, \Omega^1)$ has non-trivial torsion. Hence $H^1(\overline{X}, \Omega^1)$ is at least two-dimensional. On the other hand the Hodge-de Rham spectral sequence degenerates at the $E_1$-page since $\overline{X}$ lifts to $W_2$, and we conclude that the Hodge diamond is as shown in (\ref{eqn:godeauxhodgediamond}).

\end{proof}

\section{The CvS Calabi-Yau threefolds}
We recall the construction of some remarkable Calabi-Yau threefolds given by Cynk-van Straten \cite{cvs}. The general class to which these Calabi-Yau threefolds belong is known as the class of ``double octic" Calabi-Yau threefolds, and arise as (resolutions of) double covers  $X\rightarrow \mb{P}^3$ branched along a degree eight hypersurface $D$. It is straightforward to check that eight is precisely the degree for the Calabi-Yau condition $K_X=0$. The cover of $\mb{P}^n$ is smooth if the divisor $D$ is, and will have singularities whenever $D$ does as well. The latter is the case of interest since  $D$ will be taken to be a union of eight hyperplanes in special positions. In this case we will have to blow up along double and triple lines (i.e. lines lying on the intersection of two or three hyperplanes, respectively) and also fourfold and fivefold points (i.e. points lying on the intersection of four or five hyperplanes, respectively).

The variety $X_5$ in characteristic 5 is constructed by taking the hypersurface to be the following union of eight hyperplanes
\[
(x-t)(x+t)(y-t)(y+t)(z-t)(z+t)(x + y + Az - At)(x- By- Bz+ t)=0,
\]
where $A,B $ are the two solutions to the equation
\[
x^2+x-1=0.
\]
These hyperplanes intersect at 28 double lines (all pairwise intersections) and 9 fourfold points (for completeness we list these here: the fourfold points are given by  
\begin{align*}
(x,y,z,t)=&(2A+1, -1, -1, 1),\ (1,1,\frac{A}{A-2},1),\  (1, -1, 1,1),\  (-1, -1, 1,1), \\
&(-1, 1,1,1),\  (-1, 1, -1, 1),\  (1,0,0,0),\ (0,1,0,0),\ (0,0,1,0),
\end{align*}
as is easy to verify).  Therefore, after blowing up at these 9 fourfold points and the 28 double lines, we obtain a smooth Calabi-Yau threefold, defined over $\mb{Q}(\sqrt{5})$. In fact, the intersection pattern of the reduction of this hyperplane arrangement is exactly the same after reduction $\bmod \ \pi$, where $\pi=\sqrt{5}$, and so we obtain a smooth Calabi-Yau threefold $\mc{X}$ defined over $\mb{Z}[\frac{-1+\sqrt{5}}{2}].$ We will denote the $\pi$-adic completion of this ring by $\mb{O}$, and its field of fractions of $K$; by abuse of notation we will denote the  base change of $\mc{X} $ over $\mc{O}$ as $\mc{X}$ as well. The following theorem about this variety is proved by Cynk-van Straten in \cite{cvs}
\begin{thm}[\cite{cvs}]\label{thm:cyconditions}\hfill
\begin{enumerate}
\item The variety $\mc{X}/\mc{O}$ is smooth and is a strict Calabi-Yau threefold; in particular
\[
H^1(X_5, \mc{O})=H^2(X_5, \mc{O})=0, \Omega^3\cong \mc{O}.
\]
    \item The Hodge numbers of the generic fiber $X_K$ are given by 
    \[
    h^{11}=38, \ h^{21}=0,
    \]
    whereas the Hodge numbers of the special fiber $X_5$ are given by 
    \[
    h^{11}=39, \ h^{21}=1.
    \]
    Therefore the threefold is rigid in characteristic zero but not in characteristic 5;
    \item The unique first order deformation of $X_5$ over $\mb{F}_5[\epsilon]/\epsilon^2$ is not liftable to $\mb{F}_5[\epsilon]/\epsilon^3$.
    \item Moreover, $X_5$ does not lift to $W_2$.
    
\end{enumerate}
\end{thm}
\begin{rmk}
We remark that the jump in  the cohomology group $H^{11}$ upon reduction $\bmod \ p$, or equivalently in $H^{22}$, can be explained by the obstruction class of the deformation in the direction of the extra class in $H^{21}$.
\end{rmk}
\begin{rmk}
Note that the cohomology group $H^2(X_K, \mb{Z})$ has rank 38 since each blow up centered at the  28 lines and 9 points adds a class in $H^2$ and we start with $\mb{P}^3$ which has $H^2\cong \mb{Z}$. In other words, all the classes in $H^2(X_K)$ comes from the base of the double covering map.
\end{rmk}
There is a similar construction of the variety $Y_3$ in characteristic 3; the difference with the previous case is that the Calabi-Yau threefold constructed in characteristic 0 has bad reduction at 3, and one must perform a small resolution. The details can be found in \cite[Section 5.1]{cvs}. Here we merely summarize the facts we need in what follows:
\begin{thm}[\cite{cvs}]\label{thm:cvschar3}
There is a strict Calabi-Yau threefold $Y_3$ over $\mb{F}_3$ with no lift to any ring in which  $3\neq 0$. Its Hodge numbers are given by 
\[
h^{11}=42, \ h^{21}=0.
\]
In particular, the variety $Y_3$ is rigid and does not lift to $W_2$. 
\end{thm}
\section{Supersingularity}
In this section we show that both threefolds $Y_3$, $X_5$ constructed by Cynk-van Straten are supersingular (see Definition \ref{defn:supersingular}). 

We first recall the definition of the Artin-Mazur formal group: for any variety $X$ of dimension $n$ over  an algebraically closed  field $k$ of positive characteristic, the functor $\Phi$ on Artin local $k$-algebras with residue field $k$ given by 
\[
A\mapsto \Phi(A)\defeq H^n(X\otimes_kA, \mb{G}_m)\rightarrow H^n(X\otimes_k, \mb{G}_m)
\]
is representable by a formal group when $H^i(X, \mc{O})=0$ for all $i\neq 0, n$. So in particular $\Phi$ is representable whenever $X$ is a strict Calabi-Yau variety.
\begin{defn}
For $X$ a strict CY, the one-diensional formal group representing the functor $\Phi$ is called the Artin-Mazur formal group of $X$. The height of $X$, written $\height(X)$, is defined to be the height of this formal group.
\end{defn}

Note that the height can be a positive integer or infinity; the latter occurs when the formal group is isomorphic to $\hat{\mb{G}}_a$.

\begin{defn}\label{defn:supersingular}
We say that  a Calabi-Yau threefold $X$ in characteristic $p$ is supersingular if its height is $\infty$.
\end{defn}

Now we recall another notion of height, introduced by \cite{yobuko}. For this recall that we have the Witt sheaves $W_m(\mc{O})$, as well maps Frobenius maps $F: W_m(\mc{O})\rightarrow F_*W_m(\mc{O})$ and restriction maps $R_{m-1}: W_m(\mc{O})\rightarrow  \mc{O}$. 

\begin{defn}
The quasi-Frobenius splitting height $\height^s(X)$ is the smallest postive integer $h$ such that there exists a map $\phi: F_*W_m(\mc{O})\rightarrow O$ such that the  diagram 

\[ 
\begin{tikzcd}
W_m\mc{O}_X \arrow{d}[swap]{R^{m-1}} \arrow{r}{F} & F_*W_m\mc{O}_X \arrow{ld}{\phi}  \\
\mc{O}_X  &  \\
\end{tikzcd}
\]
commutes. Note that as in the case of the Artin-Mazur height this is a positive integer or infinity.
\end{defn}

The following result was proved by Yobuko in \cite{yobuko}
\begin{thm}\label{thm:yobuko}
\hfill
\begin{enumerate}
\item For a Calabi-Yau variety over $k$, the Artin-Mazur height is equal to the the quasi-Frobenius splitting height, that is,
\[
\height(X)=\height^s(X).
\]
\item If the height is finite, then $X$ admits a lift to $W_2(k)$.
\end{enumerate}
\end{thm}
By Theorem \ref{thm:yobuko} and the fact that neither of the varieties $X_3, X_5$ lift to $W_2(k)$, we have the following result:
\begin{thm}\label{thm:ss}
The varieties $Y_3, X_5$ are both supersingular.
\end{thm}
\begin{rmk}
Theorem \ref{thm:ss} implies that the first slope of the Newton polygon is at least 1: indeed, the part of $H^3_{\crys}(X/W)\otimes K$ with slopes in the interval $[0,1)$ correspond to the $p$-divisible quotient of the Artin-Mazur formal group \cite[Corollary 3.3]{artinmazur}. Since by Theorem \ref{thm:ss} there is no $p$-divisible quotient, there must not be any slopes in the interval $[0,1)$, as required.

This leaves two possibilities for the Newton polygons for $Y_3$ and $X_5$: either having slopes $\{1, 2\}$ or slopes $\{3/2, 3/2\}$. It is in fact possible to compute the precise Newton polygon by computing that of the corresponding Hilbert modular form, since we know that $\mc{X}$ is Hilbert-modular and the explicit modular form was found by Cynk-Sch\"utt-van Straten \cite{cynkhilbert}. 
\end{rmk}
\begin{rmk}
It is also possible to give a purely computational proof of the weaker  fact that  the action of Frobenius on the middle cohomologies of $X_3, X_5$ has no unit roots. Indeed, by a result of Stienstra, for a double cover of $\mb{P}^n$ branched along a hypersurface with equation $W=0$ the trace of Frobenius $\bmod \ p$ can be computed as a certain coefficient of a power of $W$. Since we have explicit expressions for the polynomials $W$ in each case, this is easily done with the help of a computer. For example for $X_5$ we have to compute the coefficient of $t^4x^4y^4z^4$ in $W^2$, and the answer turns out to be $60A+85$, which is divisible by 5.
\end{rmk}

%\begin{equation*}
%\begin{tikzpicture}
  %\matrix (m) [matrix of math nodes,
    %nodes in empty cells,nodes={minimum width=5ex,
    %minimum height=5ex,outer sep=-5pt},
    %column sep=1ex,row sep=1ex]{
               %&   H^0(\Omega^3)     &  H^1(\Omega^3)      & %H^2(\Omega^3)  &  H^3(\Omega^3) &\\
                %& H^0(\Omega^2)      &   H^1(\Omega^2)  &H^2(\Omega^2)  %   & H^3(\Omega^2) &\\
               %&  H^0(\Omega^1) &  H^1(\Omega^1)  & H^2(\Omega^1) & %H^3(\Omega^1) & \\
               %&  H^0(\Omega^0)  & H^1(\Omega^0) &  H^2(\Omega^0)  &  %H^3(\Omega^0) &\\
%    \quad\strut &     &    &    &  &\strut \\};
%  \draw[-stealth] (m-1-2.south east) -- (m-2-4.north west);
%  \draw[-stealth] (m-3-3.south east) -- (m-4-5.north west);
%\draw[thick] (m-1-1.north east) -- (m-5-1.east) ;
%\draw[thick] (m-5-1.north) -- (m-5-5.north east) ;
%\end{tikzpicture}
%\end{equation*}

\section{Conjugate spectral sequence and liftings}
In this section we recall some facts about the conjugate spectral sequence, prove a criterion for non-degeneration of the conjugate spectral sequence in terms of liftings; as a result we will be able to answer a question of Ekedahl. 
\subsection{A criterion for non-degeneration of the conjugate spectral sequence}
We state a lemma which will be of use later on. First recall the conjugate spectral sequence (see Figure \ref{fig:css}), specialized to the case of threefolds. We have only drawn the arrows which will be of particular interest to us.

\begin{figure}[H]
\begin{tikzpicture}
  \matrix (m) [matrix of math nodes,
    nodes in empty cells,nodes={minimum width=5ex,
    minimum height=5ex,outer sep=-5pt},
    column sep=1ex,row sep=1ex]{
               &   H^0(\Omega^3)     &  H^1(\Omega^3)      & H^2(\Omega^3)  &  H^3(\Omega^3) &\\
                & H^0(\Omega^2)      &   H^1(\Omega^2)  &H^2(\Omega^2)     & H^3(\Omega^2) &\\
               &  H^0(\Omega^1) &  H^1(\Omega^1)  & H^2(\Omega^1) & H^3(\Omega^1) & \\
               &  H^0(\Omega^0)  & H^1(\Omega^0) &  H^2(\Omega^0)  &  H^3(\Omega^0) &\\};
   % \quad\strut &     &    &    &  &\strut \\};
  \draw[-stealth] (m-1-2.south east) -- (m-2-4.north west) ;
  \draw[-stealth] (m-3-3.south east) -- (m-4-5.north west);
%\draw[thick] (m-1-1.north east) -- (m-5-1.east) ;
%\draw[thick] (m-5-1.north) -- (m-5-5.north east) ;
\end{tikzpicture}
\caption{$E_2$-page of the conjugate spectral sequence}
\label{fig:css}
\end{figure}
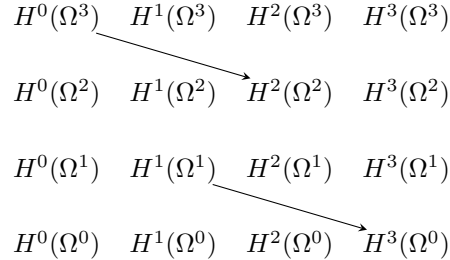

We record the following lemma, which is a direct consequence of the methods in  Deligne-Illusie \cite{deligneillusie}, and certainly well known to experts.
\begin{lem}\label{lem:diffcup}
Let $X$ be a smooth proper variety over $k$. Then each differential
\[
H^{i}(X, \Omega^j)\rightarrow H^{i+2}(X, \Omega^{j-1})
\]
in the conjugate spectral sequence (\ref{fig:css}) is given by cup product with the class $\xi\in H^2(X, \mc{T}_X)$ which is  the obstruction class to lifting $X$ over $W_2(k)$.
\end{lem}
\begin{proof}
We recall some basic facts about the conjugate spectral sequence. By definition it is the spectral sequence induced by the filtration on the de Rham complex $\Omega^{\bullet}$ given by 

\[
F^p\Omega^{\bullet}\defeq \tau_{\leq -p}\Omega^{\bullet};
\]
we denote the associated graded sheaves by $\gr^p\defeq F^p/F^{p+1}$, which in this case is given by 
\[
\gr^p=\Omega^{-p}[p].
\]
By construction, each  differential on the $E_1$-page of the spectral sequence (this is before we perform the reindexing which makes the conjugate spectral sequence into an $E_2$-spectral sequence)
\[
H^{p+q}(\gr^p)\rightarrow H^{p+q+1}(gr^{p+1})
\]
is the connecting homomorphism induced  by the sequence  
\[
0\rightarrow \gr^{p+1} \rightarrow F^p/F^{p+2}\rightarrow \gr^p \rightarrow 0.
\]
Equivalently the differential $d$ is induced by a class in $\Ext^2(\Omega^{i+1}, \Omega^i)$, or equivalently a map
\[
\Omega^{i+1}\rightarrow \Omega^i[2].
\]
Note that in our case $H^{p+q}(\gr^p)=H^{q+2p}(\Omega^{-p})$. Now Deligne-Illusie \cite{deligneillusie} shows that in the case $i=0$   this map is induced by the obstruction class $\xi$: this is the combination of Theorem 3.5 and Proposition 3.3 of loc.cit.. It remains to show that the same is true for other values of $i$: that is we would like to compute the map $\alpha: \Omega^{i+1}\rightarrow \Omega^i[2]$. Now note that we have the multiplication map 
\[
\Omega^1[-i]\otimes F^0/F^2\rightarrow F^i/F^{i+2},
\]
and since the differentials for higher $i$ (resp. $i=0$) are induced by the extension $F^i/F^{i+2}$ (resp. $F^0/F^{2}$), we have that  the composition 
\begin{equation}\label{eqn:compose}
\Omega^1\otimes \Omega^i\xrightarrow{m} \Omega^{i+1}\xrightarrow{\alpha} \Omega^i[2]
\end{equation}
is given by tensoring the map  $\cdot \cup \xi: \Omega^0\rightarrow \Omega^1[2]$   by $\Omega^i$. Finally the first map $m$ in (\ref{eqn:compose}) above is surjective, and so we are done. 
\end{proof}

\begin{lem}\label{lem:w2lift}
If a Calabi-Yau threefold $X$ over a field $k$ of characteristic $p$ has no lift to $W_2(k)$, then the differenial in the conjugate spectral sequence $H^1(\Omega^1) \rightarrow H^3(\Omega^0)$ is non-trivial.  
\end{lem}

\begin{proof}
%We first recall the conjugate spectral sequence, specialized to our case of interest.

By Lemma \ref{lem:diffcup}  the  differential  is given by cupping with the  class $\xi \in H^2(T_X)$ which is the obstruction to lifting to $W_2(k)$. But now by the Calabi-Yau condition we have 
\[
\xi \in H^2(\Omega^2)=H^2(T_X),
\]
(we will continue to denote its image in $H^2(\Omega^2)$ by $\xi$) and consequently the differential in the spectral sequence is the same as the one induced by cupping with $\xi$ in the map 
\begin{equation}\label{eqn:cup}
H^1(\Omega^1)\otimes H^2(\Omega^2)\rightarrow H^3(\Omega^3)\cong H^3(\Omega^0),
\end{equation}
where the last isomorphism is again by the Calabi-Yau condition.
On the other hand by Serre duality the pairing (\ref{eqn:cup}) is non-degenerate, and hence the differential in question is non-trivial as claimed.
\end{proof}

\subsection{The Hirokado threefold}
The Hirokado threefold $\mc{H}$ was constructed  by Hirokado \cite{hirokado}  as the first example of a Calabi-Yau variety in positive characteristic which does not admit a lift to characteristic zero; the $p$-adic cohomologies (for $p=3$, since the Hirokado threefold is in characteristic three) were further investigated by Ekedahl in \cite{ekedahl} (this variety was denoted by $X$ by both Hirokado and Ekedahl, but we have opted to denote it by $\mc{H}$ so as not to confuse it with the other varieties considered in this paper). However in loc.cit. the computation of the crystalline cohomology could not be  completed: we state the following theorem of Ekedahl:

\begin{thm}[Theorem 4.2 of \cite{ekedahl}]\label{thm:ekedahl}\hfill
\begin{enumerate}
    \item The Hirokado threefold $\mc{H}$ does not admit a lift to $W/9$;
\item 
the crystalline cohomology of $X$ is given by 
\begin{align}\begin{split}
&H^0_{\crys}(\mc{H}/W) = H^6_{\crys}
(\mc{H}/W) = W,\\
&H^1_{\crys}
(\mc{H}/W) = H^5_{\crys}
(\mc{H}/W) = 0, \\
&H^2_{\crys}(\mc{H}/W) = W^{41}, \ H^3_{\crys}
(\mc{H}/W) = W/p^nW, \\
\text{and} \ &H^4_{\crys}
(\mc{H}/W) =
W/p^nW\oplus W^{41}.
\end{split}
\end{align}
Furthermore the Hodge-de Rham spectral
sequence  degenerates at the $E_1$-page  precisely when $n > 0$.
\end{enumerate}
\end{thm}

The precise value of $n$ was not determined, however. Here we show that Lemma \ref{lem:w2lift} implies that $n=0$, which completes the computation of all cohomologies of $X$ (including the de-Rham Witt cohomology, which was also studied in detail in \cite{ekedahl}).

\begin{thm}
The crystalline cohomology of $\mc{H}$ is torsion free: that is, the value of $n$ in Theorem \ref{thm:ekedahl} is zero.
\end{thm}
\begin{proof}
By Theorem \ref{thm:ekedahl}, $\mc{H}$ does not lift to $W/9$, and hence by Lemma \ref{lem:w2lift} the differential $H^1(\Omega^1)\rightarrow H^3(\Omega^0)$ in the conjugate spectral sequence is non-zero. In particular, the conjugate spectral sequence does not degenerate at $E_2$, and hence the Hodge-de Rham spectral sequence does not degenerate at $E_1$. Therefore by the last part of Theorem \ref{thm:ekedahl} we have $n=0$, as claimed.
\end{proof}

%\begin{question}
%From Theorem \ref{thm:hirokado} we see that the de Rham, Betti and crystalline cohomologies of $\mc{H}$ all agree with those of the Cynk-van Straten threefold $Y_3$. Are these varieties isomorphic, or at least birational?
%\end{question}

\section{Topology}\label{section:topology}
In this section let $\mc{X}/\mc{O}$ denote the Cynk-van Straten threefold over the ring $\mc{O}=\mb{Z}\big[\frac{-1+\sqrt{5}}{2}\big]$, and $X$ its generic fiber. The goal of this section is to prove Proposition \ref{prop:topology} below, which will be crucial when we compute the cohomologies of the special fiber $X_5$.
\begin{prop}\label{prop:topology}
The first homology  $H_1(X, \mb{Z})$  is 2-torsion (which may or may not be vanishing). 
%%%%%%%%%%%%%
%\todo{should figure this out probably}. 
In particular, w have $H_1(X, \mb{F}_5)=0$. 
\end{prop}
\begin{proof}
Recall that $X$ is constructed by first blowing up $\mb{P}^3$ at 28 double curves and 9 fourfold points to get the threefold $Y$, and then taking the double cover of $Y$ branched along the smooth divisor $D$, which is the strict transform of the original collection of eight hyperplanes.

Now $H_1(Y, \mb{Z})=0$ since it comes from blowing up points and lines so that the exceptional divisors have no $H_1$. Now consider the double cover $\pi: X\rightarrow Y$ ramified along $D\subset Y$.

For each prime $l\neq 2$, we will compute $H_1(X, \mb{F}_l)$ and $H_1(Y, \mb{F}_l)$ using the Mayer-Vietoris exact sequence.  Since the first Betti number of $X$ is zero, we will have proven the  proposition if we can show that $H_1(X, \mb{F}_l)=0$.

Note that the preimage $\pi^{-1}(D)$ is isomorphic to $D$ since the cover is branched along $D$, so we will abuse notation and write $D$ for this preimage inside $X$ as well. Now for our application of Mayer-Vietoris we write  
\begin{align}
    X=(X\setminus D)&\cup D_{\epsilon}',\\
    Y=(Y\setminus D)&\cup D_{\epsilon},
\end{align}
where $D_{\epsilon}'$ (respectively $D_{\epsilon}$) is a tubular neighborhood of $D$ inside $X$ (respectively $Y$). Also denote by the intersection of $X\setminus D$ and $D_{\epsilon}'$ (respectively the intersection of $Y\setminus D$ and $D_{\epsilon}$) by $\Int'$  (respectively $\Int$). We then have the following  exact sequences from Mayer-Vietoris (with $\mb{F}_l$-coefficients which we omit from the notation):

\begin{equation}\label{eqn:MV}
\begin{tikzcd}
  \cdots \arrow[r] & H_1(\Int') \arrow[d, "\alpha"] \arrow[r, "f'"] & H_1(X\setminus D)\oplus H_1(D_{\epsilon}') \arrow[d, "\beta"] \arrow[r, "g'"] & H_1(X) \arrow[d, "\gamma"] \arrow[r] & H_0(\Int') \arrow[d] \arrow[r] &\cdots\\
  \cdots  \arrow[r] & H_1(\Int) \arrow[r, "f"] & H_1(Y\setminus D)\oplus H_1(D_{\epsilon}) \arrow[r, "g"] & H_1(Y) \ar[r] & H_0(\Int)\arrow[r]&\cdots 
\end{tikzcd}
\end{equation}

Note that $H_1(D_{\epsilon}')=H_1(D_{\epsilon})=0$, since both $D_{\epsilon}'$ and $D_{\epsilon}$ deformation retract onto $D$ which is a disjoint union of $\mb{P}^2$'s. Now any torsion in $H_1(X, \mb{Z}) $ must come from $H_1(X\setminus D)$, since $H_0(\Int')$ is torsion free. Note also that $H_1(Y)=0$ since it is the blow up of $\mb{P}^3$ at points and lines, and all the exceptional divisors have vanishing $H_1$.

So we can restrict attention to the following part of diagram \ref{eqn:MV}:
\begin{equation}\label{eqn:sqmv}
    \begin{tikzcd}
     H_1(\Int')\arrow[r, "f'"] \arrow[d, "\alpha"] & H_1(X\setminus D)\arrow[d, "\beta"]\\
     H_1(\Int)\arrow[r, "f"]  & H_1(Y\setminus D)
    \end{tikzcd}
\end{equation}
By construction the map $\pi: X\setminus D \rightarrow Y\setminus D$ is an unramified double cover and therefore the maps $\alpha$  and $\beta$ are both  isomorphisms as long as we use homology with $\mb{F}_l$-coefficients for $l\neq 2$.  

Since the bottom horizontal map in \ref{eqn:sqmv} is surjective,   $f'$ is as well, and therefore $H_1(X, \mb{F}_l)=0$, as required. 

\end{proof}

\section{The CvS variety in characteristic 3}\label{section:cvs3}
\begin{thm}\label{thm:cvsy3}
The Calabi-Yau threefold $Y_3$ has $H^3_{crys}=0$, and hence vanishing third Betti number.
\end{thm}
\begin{proof}
Recall that for $Y_3$ the Hodge numbers in degree three are given by 
\[
h^{30}=h^{03}=1, \ h^{21}=h^{12}=0.
\]
Now since $Y_3$ does not lift to $W_2$ by Theorem \ref{thm:cvschar3}, we may apply  Lemma \ref{lem:w2lift} to $Y_3$, and deduce that the differentials
\begin{align}
H^1(Y_3, \Omega^1) &\rightarrow H^3(Y_3, \Omega^0)\\
H^0(Y_3, \Omega^3) &\rightarrow H^2(Y_3, \Omega^2)
\end{align}
are non-zero; therefore   $H^3_{\dR}(Y_3)=0$ since the two  classes in $H^3(Y_3, \Omega^0)$ and $H^0(\Omega^3)$ are both killed by the differentials. Therefore $H^3_{\crys}=0$ by universal coefficients.

The vanishing of $b_3$ now follows from a  theorem of Katz-Messing (and does not require liftability): indeed as a consequence of the Weil conjectures they proved \cite[Theorem 1]{katzmessing} that Frobenius on rational cystalline cohomology and $l$-adic \'etale cohomologies have the same characteristic polynomials (for any $l$), and so in particular $H^i_{\crys}[1/p]$ and $H^i_{\acute{e}t}$ have the same dimensions.
\end{proof}

\begin{rmk}
This addresses a question in \cite[p.1]{takayama} about the third Betti number of $Y_3$ (see ``The author is not aware..."). Note that this does not give a new proof of non-liftability of $Y_3$ since we used the non-liftability  to deduce the vanishing of $b_3$.
\end{rmk}

%\begin{rmk}
%This is in accordance with Joshi's conjecture, since the condition that $b_3\neq 0$ is equivalent to being "non-classical" in the language of \cite{joshi}.
%\end{rmk}
\begin{question}
We note that the Hodge numbers $h^{21}=0$ and $h^{11}=42$ of the CvS threefold $Y_3$ agree with the corresponding Hodge numbers  of the Hirokado threefold $\mc{H}$ (see Theorem \cite[Theorem 3.6 (iii)]{ekedahl}): in fact, all the Hodge numbers of $Y_3$ and $\mc{H}$ agree, with the possible exception of $h^{10}$ and $h^{20}$ (and the Serre dual of these, of course). Also,  by Theorem \ref{thm:cvsy3} above, the crystalline cohomologies $H^3_{\crys}$ also agree. Finally, using the same techniques as Section \ref{section:topology}, one may compute  all the Betti numbers of $Y_3$  as well and see that they agree with those of $\mc{H}$ (we may compute the Betti numbers of $\mc{H}$ from the crystalline cohmologies \ref{thm:ekedahl} and  the result of Katz-Messing\cite[Theorem 1]{katzmessing}.
%also constructed a non-liftable Calabi-Yau threefold $\mc{F}$ in characteristic 3 with the same Hodge numbers as $Y_3$. 
We therefore find it natural to ask: is $Y_3$ isomorphic to  $\tilde{\mc{F}}$, or at least birational to it? 
\end{question}

\section{The CvS variety in characteristic 5}
We first recall the following theorem in integral $p$-adic Hodge theory due to Caruso \cite[Th\'eor\`eme 1.1]{caruso} and Faltings \cite{faltings} (see also \cite[Remark 2.3 (4)]{bhattlecture}):

\begin{thm}\label{thm:carusofaltings}
For a variety $\mc{X}$ over a ring of integers $\mc{O}$ of some $p$-adic field $K$ with good reduction, let $e$ be the ramification degree of $K/\mb{Q}_p$, and $k$ be the residue field of $\mc{O}$. If $ie<p-1$, then we have %\todo{check reference}
\[
\dim_{k}H^i_{\dR}(X_k)=\dim_{\mb{F}_p}H^i(X_{\acute{e}t},\mb{F}_p).
\]

\end{thm}

\begin{lem}\label{lem:h10}
For the CvS Calabi-Yau threefold $X_5/\mb{F}_5$,  we have $H^0(\Omega^1)=0$.

\end{lem}
%\begin{rmk}
%Note that this is in accordance with Joshi's conjecture (in the direction  vanishing of $H^0(\Omega^1)$ implies liftability), since $X_5$ certainly lifts to characteristic 0.
%\end{rmk}
We give two proofs of this.
\begin{proof}[Proof 1]
We again consider the conjugate spectral sequence. 
\begin{center}
\begin{tikzpicture}
  \matrix (m) [matrix of math nodes,
    nodes in empty cells,nodes={minimum width=5ex,
    minimum height=5ex,outer sep=-5pt},
    column sep=1ex,row sep=1ex]{
               &   H^0(\Omega^3)     &  H^1(\Omega^3)      & H^2(\Omega^3)  &  H^3(\Omega^3) &\\
                & H^0(\Omega^2)      &   H^1(\Omega^2)  &H^2(\Omega^2)     & H^3(\Omega^2) &\\
               &  H^0(\Omega^1) &  H^1(\Omega^1)  & H^2(\Omega^1) & H^3(\Omega^1) & \\
               &  H^0(\Omega^0)  & H^1(\Omega^0) &  H^2(\Omega^0)  &  H^3(\Omega^0) &\\};
   % \quad\strut &     &    &    &  &\strut \\};
  \draw[-stealth] (m-3-2.south east) -- (m-4-4.north west);
  %\draw[-stealth] (m-3-3.south east) -- (m-4-5.north west);
%\draw[thick] (m-1-1.north east) -- (m-5-1.east) ;
%\draw[thick] (m-5-1.north) -- (m-5-5.north east) ;
\end{tikzpicture}
\end{center}
If $x\in H^0(\Omega^1)$ is non-zero, then since by Theorem \ref{thm:cyconditions}, the group $H^2(\Omega^0)$ is zero, this class $x$ survives to the $E_{\infty}$-page of the spectral sequence, and hence $H^1_{\dR}(X_5)\neq 0$. Now by the integral comparison theorem of Caruso \cite{caruso} and Faltings \cite{faltings}
\[
\dim_{\mb{F}_5}H^1_{dR}(X_5)=\dim_{\mb{F}_5}H^1(X_{\acute{e}t}, \mb{F}_5)
\]
and therefore $H^1_{\Betti}(X, \mb{F}_p)\neq 0$. Since $H^1_{\Betti}(X, \mb{Z})$ vanishes (indeed, it cannot have torsion by the universal coefficients theorem, and $H^1(X, \mc{O})=0$ implies its free part vanishes also), the class in $H^1_{\Betti}(X, \mb{F}_p)\neq 0$ comes from $H^2(X, \mb{Z})$, and so by the universal coefficients theorem we have $H_1(X, \mb{F}_5)\neq 0$, contradicting Proposition \ref{prop:topology}.

\end{proof}
\begin{proof}[Proof 2]
Recall that, for the generic fiber, we have 
\[
\dim H^{1}(X, \Omega^1)=38, \ \dim H^{1}(X, \Omega^2)=0,
\]
while in the special fiber we have 
\[
\dim H^{1}(X_5, \Omega^1)=39,\  \dim H^{1}(X_5, \Omega^2)=1.
\]
This implies also that the dimensions of $H^2(\Omega^2)$ and $H^2(\Omega^2)$ go up by one upon reducing modulo $\pi$. 
This implies that there is torsion in the integral Hodge cohomology groups. Indeed, by universal coefficients, for all $i, j$ we have an exact sequence
\[
0\rightarrow H^i(\mc{X}, \Omega^j)/\pi H^i(\mc{X}, \Omega^j)\rightarrow H^i(X_5, \Omega^j)\rightarrow H^{i+1}(\mc{X}, \Omega^j)[\pi]\rightarrow 0
\]
coming from the exact sequence of sheaves 
\[
0\rightarrow \Omega^j_{\mc{O}}\xrightarrow{\cdot \pi} \Omega^j_{\mc{O}}\rightarrow \Omega^j\rightarrow 0.
\]
Here $\pi$ denotes a uniformizer of $\mc{O}$, which in our case can be chosen to be $\sqrt{5}$.
%\end{proof}

Now suppose $H^0(X_5, \Omega^1)$ is non-zero. Then by Serre-duality we have $H^3(X_5, \Omega^2)$ nonzero as well, and so there must be torsion in $H^3(\mc{X}, \Omega^2)$ (since otherwise we would have a class in $H^4(X_5, \Omega^2)$, which is absurd). Let us call this class $\xi$. Then $\xi$ contributes to the extra class in $H^2(X_5, \Omega^2)$. 

Now consider $H^1(X_5, \Omega^2)$. Again since the dimension of this group is one bigger than that of the generic fiber, there is torsion in either $H^1(\mc{X}, \Omega^2)$ or $H^2(\mc{X}, \Omega^2)$. But the latter cannot happen, since this would mean that $H^2(X_5, \Omega^2)$ has dimension at least 2 more than the generic fiber. Hence we have torsion in $H^1(\mc{X}, \Omega^2)$.

But this is the formal deformation space of $\mc{X}$, and so this implies that there is a lift of the non-trivial deformation of $X_5$ over $\mb{F}_5[\epsilon]/\epsilon^2$ to $\mc{O}[\epsilon]/\epsilon^2$. The deformations are given by equisingular deformations of the hyperplane arrangement $D$, and so we must find a solution to the equation
\[
x^2+x-1=0
\]
in $\mc{O}[\epsilon]/\epsilon^2$ lifting the solution $2+\epsilon \in \mb{F}_5[\epsilon]$. We now check that this is impossible: indeed if we have such a solution $\lambda +\mu \epsilon$ for $\lambda, \mu \in \mc{O}$, then reducing modulo $\epsilon$ we must have that
\[
\lambda=\frac{-1+\sqrt{5}}{2},
\]
and so the equation $(\lambda+\mu\epsilon)^2+(\lambda+\mu\epsilon)-1=0$ reduces to 
\[
2\lambda \mu +\mu =0
\]
which has no solutions. Therefore $H^1(\mc{X}, \Omega^2)$ cannot have torsion, and so $H^0(X_5, \Omega^1)=0$, as claimed.
\end{proof}

\begin{thm}\label{thm:cvs5crystorsionfree}
The crystalline cohomology of $X_5$ is torsion free, i.e.
\[
H^2_{\crys}(X_5/W)=W^{\oplus38}, \ H^3_{\crys}(X_5/W)=W^{\oplus 2}.
\]
Moreover the Hodge cohomologies $H^i(\mc{X}, \Omega^j)$ of the integral model  are torsion free except for $H^2(\mc{X}, \Omega^1)$ and $H^2(\mc{X}, \Omega^2)$.

\end{thm}
\begin{rmk}
Note that this gives a proof of the fact that $H^3(X, \mb{Z})$ has no 5-torsion by the recent result of Bhatt-Morrow-Scholze which says that torsion in Betti cohomology forces torsion in crystalline cohomology.  It is probably possible to compute this cohomology group directly and see that it is 5-torsion free, however. Therefore this variety has no 5-torsion in its Brauer group, which is another invariant that is often of interest in the study of Calabi-Yau manifolds \cite{grossbrauer}.
\end{rmk}
\begin{proof}
We give two arguments that $H^3_{\crys}$ is torsion free, one using the fact that $X_5$ has obstructed deformation in characteristic $5$, and one using the fact that it does not lift to $W_2$.  For the first argument, a theorem of Ekedahl-Shepherd-Barron says that for a Calabi-Yau threefold, if  

\[
\dim H^n_{\dR}=\dim \bigoplus_{i+j=n}H^i(\Omega^j),
\]
(in other words that the differentials in the non-trivial differentials in the Hodge-de Rham spectral sequence do not touch the middle cohomology, then the deformation ring of $X_5$ in characteristic 5 is height 1 smooth, which would mean that the non-trivial first order deformation lifts to $\mb{F}_5[\epsilon]/\epsilon^5$, contradiction. Therefore $H^3_{\dR}(X_5)$ has dimension 2, and hence $H^3_{\crys}$ is torsion free, as required. 

For the second argument, we use the fact that $X_5$ does not lift to $W_2$. Then by Lemma \ref{lem:w2lift} the differential from $H^1(X_5, \Omega^1)$ to $H^3(X_5, \Omega^0)$ in the conjugate spectral sequence is non-trivial, and hence the unique class in $H^3(\Omega^0)$ dies in the $E_{\infty}$-page, and so again $H^3_{\dR}({X_5})$ has dimension 2, and we argue as above.

Now we come to $H^2_{\crys}(X_5/W)$. Torsion in this group would imply $H^1_{\dR}(X_5)$ is non-zero, and hence one of $H^0(X_5, \Omega^1)$ and $H^1(X_5, \Omega^0)$ has to be non-zero. This contradicts \ref{lem:h10} and the fact that $H^1(X_5, \mc{O})=0$. 

Now we must have $H^0(X_5, \Omega^2)=0$ as well arguments similar to either proofs of \ref{lem:h10}: either use the fact that any class in $H^0(X_5, \Omega^2)=0$ must survive to the $E_{\infty}$-page of the conjugate spectral sequence, since we have $\dim_{\mb{F}_5}H^3_{\dR}(X_5)=2$ and we already have non-zero differentials $H^0(X_5, \Omega^3)\rightarrow H^2(X_5, \Omega^2)$ and $H^1(X_5, \Omega^1)\rightarrow H^3(\Omega^0)$, or use a similar argument to the second proof of Lemma \ref{lem:h10} by chasing the torsion in the integral Hodge cohomologies. 

Finally the assertions about torsion in the integral Hodge cohomologies $H^i(\mc{X}, \Omega^j)$ follow from the fact that $H^0(X_5, \Omega^1)=H^0(X_5, \Omega^2)=0$.
\end{proof}

\begin{rmk}
Note that this gives a counterexample to an expectation stated by Joshi in \cite[page 8, paragraph beginning ``We expect that..."]{joshi}, namely that for a Calabi-Yau variety $X$ over a finite field of characteristic $\geq 5$, if $b_3\neq 0$ ("classical" in the terminology of \cite{joshi}) and $H^*_{\crys}(X/W)$ is torsion free then $X$ lifts to $W_2$.
\end{rmk}

We may now sum up this section by giving the proof of Theorem \ref{thm:cvsintro}, whose precise form we now state. 

\begin{thm}\label{thm:cvsfinal} For the CvS Calabi-Yau threefold $\mc{X}/\mc{O}$ the following hold: 
\begin{enumerate}
    \item The variety $X_5/\mb{F}_5$ has torsion free crystalline cohomology; in particular the cohomology groups $H^*_{\dR}(\mc{X})$ of the integral model are torsion free as well. More precisely,  cohomology groups are given by
    \begin{align*}
    H^*_{\dR}(\mc{X}/\mc{O})&=\mc{O},\ 0,\ \mc{O}^{38},\ \mc{O}^2,\  \mc{O}^{38},\ 0,\ \mc{O};\\
    H^*_{\crys}(X_5/W)&= W,\  0,\  W^{38},\  W^2,\  W^{38},\  0,\  W.
    \end{align*}
    Here, we have written the groups $H^i_{\bullet}$ in increasing order of from $i=1$ to $i=6$. Therefore the de Rham cohomology groups of the special fiber are given by:
    \[
    H^*_{\dR}(X_5)= \mb{F}_5, \ 0, \ \mb{F}_5^{38}, \ \mb{F}_5^2,\  \mb{F}^{38}, \ 0, \ \mb{F}_5.
    \]
    
    \item The Hodge cohomology groups $H^i(\mc{X}, \Omega^j)$ are torsion free $\mc{O}$-modules, except for $H^2(\mc{X}, \Omega^1)$ and $H^2(\mc{X}, \Omega^2)$, which have  non-trivial  torsion; more precisely they are given by 
    \[
    H^2(\mc{X}, \Omega^1)=\mc{O}/\pi^a, \ H^2(\mc{X}, \Omega^2)=\mc{O}^{38}\oplus \mc{O}/\pi^b,
    \]
    for some integers $a, b>0$.
    %\item Hodge symmetry holds for $X_5$: more precisely, the Hodge numbers are given by 
    %\[
    %h^{30}=1, \ h^{21}=1, \ h^{10}=h^{20}=h^{01}=h^{02}=0,
    %\]
    %and the rest are computed by Serre duality.
    \item The Hodge diamond of $X_5$ is given by 
\begin{equation*}
\begin{array}{ccccccc}
&&& 1 &&&\\
&& 0 && 0 &&\\
& 0 && 39 && 0 &\\ 
1 && 1 && 1 && 1\\
& 0 && 39 && 0 &\\ 
&& 0 && 0 &&\\
&&& 1 &&&.\\
\end{array}
\end{equation*}
%This should be compared with that of the generic fiber $X/K$, which was computed by Cynk-van Straten \cite[Section 5.2]{cvs}:  
%\begin{equation*}
%\begin{array}{ccccccc}
%&&& 1 &&&\\
%&& 0 && 0 &&\\
%& 0 && 38 && 0 &\\ 
%1 && 0 && 0 && 1\\
%& 0 && 38 && 0 &\\ 
%&& 0 && 0 &&\\
%&&& 1 &&&.\\
%\end{array}
%\end{equation*}
Note in particular that Hodge symmetry holds for $X_5$, which is not generally the case for smooth proper varieties over finite fields.
%\item The dimensions of the de Rham cohomology groups are given by 
%\[
%\dim H^*_{\dR}(X/k)= 1, \ 0,\  38,\  2,\  38,\  0,\  1.
%\]

    \item The special fiber $X_5$ is supersingular, i.e. its  Artin-Mazur formal group is isomorphic to $\hat{\mb{G}}_a$.
\end{enumerate}

\end{thm}

Even though some of the arguments have already been given  in the course of the proofs of various theorems, we give them again here for the convenience of the reader.
\begin{proof}[Proof of Theorem \ref{thm:cvsfinal}]
For the computation of the cohomology groups $H^*_{\dR}(X_5)$, we have the  $E_2$-page of the conjugate spectral sequence  
\begin{center}
\begin{tikzpicture}
  \matrix (m) [matrix of math nodes,
    nodes in empty cells,nodes={minimum width=5ex,
    minimum height=5ex,outer sep=-5pt},
    column sep=1ex,row sep=1ex]{
               &   \mb{F}_5    &  0     & 0 &  \mb{F}_5  &\\
                & 0     &   \mb{F}_5   &\mb{F}_5^{39}      & 0 &\\
               &  0 &  \mb{F}_5^{39}   & \mb{F}_5 & 0 & \\
               &  \mb{F}_5  & 0 &  0  &  \mb{F}_5, &\\};
   % \quad\strut &     &    &    &  &\strut \\};
  \draw[-stealth] (m-3-3.south east) -- (m-4-5.north west);
  \draw[-stealth] (m-1-2.south east) -- (m-2-4.north west);
  %\draw[-stealth] (m-3-3.south east) -- (m-4-5.north west);
%\draw[thick] (m-1-1.north east) -- (m-5-1.east) ;
%\draw[thick] (m-5-1.north) -- (m-5-5.north east) ;
\end{tikzpicture}
\end{center}
where we have drawn the only non-trivial differentials. This gives 
\[
H^*_{\dR}(X_5)=\mb{F}_5, \ 0, \ \mb{F}_5^{38}, \ \mb{F}_5^2, \ \mb{F}_5^{38}, \ 0, \ \mb{F}_5
\]
as claimed. Since $H^*_{\crys}(X_5/W)$ is torsion free by Theorem  \ref{thm:cvs5crystorsionfree},  $H^*_{\dR}(\mc{X})$ is also torsion free, which  proves part (1) of the theorem. By universal coefficients and comparing Hodge cohomologies of the generic fiber and those of the special fiber, we must have a torsion class in $H^2(\mc{X}, \Omega^2)$ giving rise to extra classes in $H^2(X_5, \Omega^2)$ and $H^1(X_5, \Omega^2)$; similarly there must be a torsion class in $H^2(\mc{X}, \Omega^1)$ giving extra classes in $H^2(X_5, \Omega^1)$ and $H^1(X_5, \Omega^1)$. Finally the torsion parts of $H^2(\mc{X}, \Omega^2)$ and $H^2(\mc{X}, \Omega^1)$ are one dimensional upon reducing mod $\pi$, so we have proven part (2).

We have computed the Hodge diamond of $X_5$, and the supersingularity was proved in \ref{thm:ss}: this gives parts (3) and (4). This concludes the proof of the theorem.
\end{proof}

\section{Further questions}\label{section:further}
In this section we discuss some future directions and possible implications for the main theme of this paper, namely the relationship between torsion in cohomology of the generic fiber and various cohomologies of the special fiber.
\subsection{Hodge numbers and derived equivalences}\label{section:hodgenumber}
The derived category $D^b_{coh}(X)$ of a variety $X$ has become a prominent player in recent years in algebraic geometry, partly due to its role in the homological mirror symmetry conjectures. One of the major questions in the study of this category is how much information is contained in $D^b_{coh}(X)$; for example, one could ask about numerical invariants such as the Hodge numbers of $X$. For smooth proper varieties of dimension up to 3 in characteristic zero, it is known by work of \cite{popaschnell} that the Hodge numbers are derived invariants: more precisely, for varieties $X, Y$ for which there exists an equivalence of categories
\[
D^b_{coh}(X)\cong  D^b_{coh}(Y),
\]
we have $h^{ij}(X)=h^{ij}(Y)$ for all $i,j$. This question was investigated by Antieau-Bragg in \cite{antieaubragg} for varieties in positive characteristic, where they prove the same result for varieties up to dimension 2, and for threefolds under some restrictions; in particular for threefolds they assume that the crystalline cohomologies are torsion free.

Here we suggest that the Hodge numbers should not be derived invariant. We suggest several potential counterexamples, and hope to check the details in future work. OUr idea is simply that, as has been exploited several times in this paper, torsion in Betti cohomology in the generic fiber contributes to Hodge cohomologies in the special fiber. More precisely, for derived equivalent threefolds  $X$ and $Y$, if the torsion in their cohomologies differ, then it is very likely that their Hodge cohomologies differ as well. 

More specifically, such examples have been shown to exist by Addington \cite{addington}. Indeed, in loc.cit. two derived equivalent Calabi-Yau threefolds  $X$
 and $Y$ are exhibited such that 
\begin{align}\label{eqn:addington}
\begin{split}
&\pi_1(X)\cong 0, \ \ \rm{Br}(X)\cong (\mb{Z}/8)^{\oplus 2},\\
    &\pi_1(Y)\cong (\mb{Z}/8)^{\oplus 2}, \ \ \rm{Br}(Y)\cong 0.
    \end{split}
\end{align}
Here, $\rm{Br}$ denotes the Brauer group, which for a Calabi-Yau threefold $X$ simply means $H^3(X,\mb{Z})_{\rm{tors}}$. The derived equivalence between $X$ and $Y$ was shown  by Bak \cite{bak} as well as \cite{schnell}, and Addington  found that a certain extension between $H_1$ and $\rm{Br}$, thereby showing (\ref{eqn:addington}). 

The Calabi-Yau threefolds $X$ and $Y$ were originally studied by Gross-Popescu \cite{grosspopescu} and the Brauer group was computed by Gross-Pavanelli \cite{grosspavanelli}; one feature is that they are fibered by abelian surfaces. It seems likely that one can construct these threefolds in characteristic 2, and compute their cohomologies explicitly. For our purposes it is more convenient to have primes bigger than 2 in the torsion; for this we will consider analogous abelian surfaces fibered Calabi-Yau threefolds studied by Donagi-Gao-Schulz \cite{donagijunction} (though it is  likely that these have also appeared in \cite{grosspopescu}), where they compute torsion in the cohomologies of these threefolds. In particular,  the (family of) varieties $\mc{X}_{1,3}$, $\mc{X}_{3,1}$ (following  the notation of loc.cit.) have 
\begin{align*}
&H_1(\mc{X}_{3,1})\cong 0,\ \ \rm{Br}(\mc{X}_{3,1})\cong (\mb{Z}/3)^{\oplus 2},\\
&H_1(\mc{X}_{1,3})\cong (\mb{Z}/3)^{\oplus 2};\ \ \rm{Br}(\mc{X}_{1,3})\cong 0,\\
\end{align*}
furthermore it seems likely that they are derived equivalent, as in the case considered by Addington above. We state the following conjecture, which would easily imply that Hodge numbers are not derived invariant in positive characteristic (using Caruso's result), and we will study it in  future work.

\begin{conj}\label{conj:hodgenumber}
The Calabi-Yau threefolds in \cite{donagijunction} can be constructed rigorously, and (generically) have good reduction at  primes $p$ dividing the torsion in its Betti cohomology. Furthermore, the derived equivalences between these Calabi-Yau threefolds %(shown to exist in \cite{donagijunction} at a physical level of rigor) 
extend to characteristic $p$.
\end{conj}

%Therefore, if $\mc{X}_{1,3}$ and $\mc{X}_{3,1}$ can be found with good reduction at $3$, and the derived equivalence extends to the special fiber, then the special fibers would have different Hodge numbers. Indeed, let $\mc{M}_{1,3}$ (resp. $\mc{M}_{3,1}$) denote the moduli space of $\mc{X}_{1,3}$ (resp. $\mc{X}_{3,1}$, a generic choice of $\mb{F}_q$ point of    $\mc{M}_{1,3}$, say $X$, will lift to $W_2(\mb{F}_q)$; assuming that the derived equivalences extend to characteristic $3$, we can find a point $Y$ on $\mc{M}_{3,1}$ which is derived equivalent to $X$, also admitting a lift to $W_2$\footnote{This can either be done by a generic choice of $X$, or by constructing a lift of $X$ to $W_2$ and finding its Fourier-Mukai partner explicitly, which should be by taking the relative Picard scheme over the base of the fibration}. The theorem of Caruso \ref{thm:carusofaltings} implies that $H^1_{\dR}(X)\neq 0$, as well as that $H^1_{\dR}(Y)=0$. Since they both admit lifts to $W_2$,   $\dim H^1_{\dR}$ is just the sum of the Hodge numbers in degree 1, so we have exhibited derived equivalent threefolds with different Hodge numbers.

\subsection{Torsion in cohomology and deformations}\label{section:torsiondeformation}
As emphasized at several points of this paper, one of our main points was to leverage the relation between the Betti cohomology of the generic fiber of a  variety and various cohomologies of the special fiber. Thinking along these lines, we have the following question which we hope to address in future work. A Calabi-Yau threefold $X$ in characteristic zero sometimes has non-trivial Brauer group, which is to say torsion in $H^3(X, \mb{Z})$: this invariant is of interest in algebraic geometry as well as in physics. Suppose that we have $p$-torsion in $H^3(X, \mb{Z})$, and that $X$ has good reduction at $p$; let us denote its special fiber by $X_p$. Then $H^3_{\dR}(X_p)$ will have extra classes, and therefore also in Hodge cohomologies; assuming $X_p$ is still CY, the extra classes will appear in $H^1(X_p, \Omega^2)$ (as well as in $H^1(X_p, \Omega^2)$, by Serre-duality). In other words, the deformation space in characteristic $p$ now has extra dimensions. What can be said about these extra deformations? Are they obstructed? 
For examples of Calabi-Yau threefolds with non-trivial Brauer group, we refer the reader to \cite{grosspavanelli, addington}. The first example to consider seems to be the case of the Enriques Calabi-Yau threefold, which has $H^3_{\rm{tors}}\cong \mb{Z}/2$ (cf \cite[Example 3.11]{grossslag}).

In fact, we  propose the following recipe for constructing these extra directions of deformations in some cases. Let us consider again the varieties $\mc{X}_{1,3}$ and $\mc{X}_{3,1}$ from Section \ref{section:hodgenumber}. In fact, we have 
\[
\mc{X}_{1,3}=\mc{X}_{3,1}/(\mb{Z}/3\times \mb{Z}/3),
\]
where on the right hand side $\mb{Z}/3\times \mb{Z}/3$ acts freely. This immediately suggests that there is an intermediate threefold, which is of the form 
\[
Z\defeq \mc{X}_{3,1}/\mb{Z}/3,
\] 
for some choice of $\mb{Z}/3\hookrightarrow\mb{Z}/3\times \mb{Z}/3$. By the result  of Addington \cite[Proposition on page 3]{addington}, we have that 
\[
\rm{Br}(Z)=\mb{Z}/3,
\]
and so we may expect that $h^{21}$ jumps by 1 in the special fiber, giving us an extra direction of deformation. Our speculation  is that, in the case that the  groups are $\alpha_3 \times \alpha_3$  in the special fiber, we may choose a diagonal 
\[
\alpha_3\hookrightarrow \alpha_3\times \alpha_3
\]
to form the quotient variety $Z$, and that, just as in the Moret-Bailly \cite{moretbailly} example, this choice of diagonal $\alpha_3$ has moduli: there is a $\mb{P}^1$-worth of diagonal $\alpha_3$'s with which we can form the quotient $\mc{X}_{3,1}/\alpha_3$, giving us an extra direction to deform in. One can make similar speculations for the quotients associated to  each of the  $\mc{X}_{m,n}$'s considered in Section \ref{section:hodgenumber}.

On the other hand, in the example $Z$ just mentioned,  there should be extra deformations even if the group is not $\alpha_3\times \alpha_3$ in the special fiber: we are not sure what to expect in these cases.

\subsection{Modularity} 
The rigid variety $X/\mb{Q}(\sqrt{5})$ has also been shown by Cynk-van Straten \cite{cynkhilbert} to be modular, in the sense that the two dimensional Galois representation of $\Gal(\overline{\mb{Q}}/\mb{Q}(\sqrt{5})$ given by the $l$-adic \'etale cohomology of $X$ is isomorphic to that of a Hilbert modular form on $\GL_2(\mb{Q}(\sqrt{5}))$. 
\begin{question}
Does the torsion class in $H^2(\mc{X}, \Omega^1)$ have any arithmetic meaning in terms of this Hilbert modular form or the Galois representation?
\end{question}

\subsection{Cohomology of the Hirokado variety}
Ekedahl has interpreted the Hirokado variety as a Deligne-Lusztig %\todo{precise statement? parabolic type?}
(which is in turn possibly isomorphic to the CvS variety $Y_3$). The $l$-adic, $\bmod \ l$, and even crystalline  cohomologies of Deligne-Lusztig varieties are known to give representations of finite groups of Lie type. Does the fact that the Hodge-de Rham spectral sequence does not degenerate have any interesting consequences from this point of view?

\subsection{Relations between $\mc{H}$ and $Y_3$}
As observed in Section \ref{section:cvs3}, all the known  cohomologies of the Hirokado threefold $\mc{H}$ and the CvS threefold $Y_3$ agree. Are these two varieties in fact isomorphic, or at least birational?

\appendix
\section{Construction of quotients of quintic hypersurfaces}\label{appendix}
In this appendix we show the existence of smooth Calabi-Yau threefolds which are free quotients of quintic hypersurfaces in $\mb{P}^4$ by $\mu_5$. The calculation follows essentially that of Lang \cite{lang}, who constructed similar surfaces (the Godeaux surfaces) as free quoitients of quintic hypersurfaces in $\mb{P}^3$ by $\mu_5$.

Consider the action of $\mu_5$ on $\mb{P}^4$ by 
\begin{equation}\label{eqn:mu5actionp4}
(X_0: X_1:X_2:X_3:X_4)\mapsto (X_0: \zeta X_1:\zeta^2X_2:\zeta^3 X_3:\zeta^4 X_4).
\end{equation}
We prove the following 
\begin{thm}\label{thm:godeauxconstruct}
The quotient map $\mb{P}^4\rightarrow \mb{P}^4/\mu_5$ (where $\mu_5$ acts by the above action) gives an embedding of $\mb{P}^4$ into $\mb{P}^{25}$, and furthermore a generic hyperplane section of $\mb{P}^4$ under this embedding is a smooth Calabi-Yau threefold. 
\end{thm}
\begin{proof}
First we compute the dimensions of invariant quintics under the above actions of $\mu_5$. Note that the invariant polynomials are spanned by invariant monomials.  In \cite{lang} the  computation was done for the $\mu_5$-action on $\mb{P}^3$ given by the action on the last four coordinates:
\begin{equation}\label{eqn:mu5actionp3}
(X_1:X_2:X_3:X_4)\mapsto (\zeta X_1:\zeta^2X_2:\zeta^3 X_3:\zeta^4 X_4).
\end{equation}
The dimensions are encoded in the power series
\begin{equation}\label{eqn:powerseries}
1+2x^2+4x^3+7x^4+12x^5+16x^6+24x^7+\cdots;
\end{equation}
that is, the coefficient of $x^n$ is the dimension of the invariant polynomials of degree $n$.
The power series above is given by 
\[
\frac{1}{5}\big[(1-x)^{-4}+4(1-\zeta x)^{-1}(1-\zeta^2 x)^{-1}(1-\zeta^3 x)^{-1}(1-\zeta^4 x)^{-1}\big],
\]
and it is in turn a theorem of Molien that this gives the desired dimensions (cf \cite[p.421]{lang}).

Now we return to our case of interest, which is the $\mu_5$-action on $\mb{P}^4$ given by (\ref{eqn:mu5actionp4}), instead of the $\mb{P}^3$ case given by (\ref{eqn:mu5actionp3}); in what follows we will refer to these as the $\mb{P}^4$ and $\mb{P}^3$ cases respectively. Since the $\mu_5$-action  on the additional variable $X_0$ is trivial, the invariant polynomials of degree $n$ in the $\mb{P}^4$ is simply given by the invariant polnomials in degree $\leq n$ in the $\mb{P}^3$ case, multiplied by  the appropriate power of $X_0$ to make it of degree $n$.  Therefore there are 
\[
1+2+4+7+12=26
\]
invariant quintics in the $\mb{P}^4$ case.

As in \cite{lang} we show that the map 
\[
\mb{P}^4 \rightarrow \mb{P}^{25}
\]
given by these  $26$ invariant quintics is an embedding. To do this it suffices to show that the invariant degree $10$ polynomials, degree $15$ polynomials, etc, are all generated by the invariant quintics. It is easier to do this calculation in terms of the invariant polynomials in the $\mb{P}^3$ case; translated to this setting, the statement becomes that the degree $\geq 6$ invariant polynomials are all generated by those of degree $\leq 5$. We now explicitly check this. 
\begin{itemize}
    \item degree 2: $X_1X_4$, $X_2X_3$;
    \item degree 3: $X_1^2X_3$, $X_2^2X_1$, $X_3^2X_4$, $X_4^2X_2$;
    \item degree 4: $X_1^3X_2$, $X_2^3X_4$, $X_3^3X_1$, $X_4^3X_3$, $(X_1X_4)^2$, $(X_2X_3)^2$, $X_1X_2X_3X_4$. 
\end{itemize}

On the other hand, from (\ref{eqn:powerseries}) above we have that the number of invariant degree six polynomails is 16 (for the $\mb{P}^3$ case).

In degree six, none of the powers $X_i^6$ are invariant, and hence each invariant monomial has at least two variables. Suppose we have such an invariant monimal $\psi$; if it is not divisible by any cube $X_i^3$, then it is straightforward to check that it must be one of 
\[
(X_1X_4)^2X_2X_3, \ X_1X_4(X_2X_3)^2,
\]
which is certainly generated by lower degree invariants (namely the invariant degree two polynomials). On the other hand, if the invariant monomial $\psi$ is divisible by $X_i^3$, say $X_1^3$ without loss of generaility, since we have the invariants 
\[
X_1X_4, \ X_1^2X_3, \ X_1^3X_2
\]
in degrees two, three, four respectively, and $\psi$ contains at least two variables, it is divisible by a lower degree invariant monomial, as required. This shows that the degree six invariants are generated by lower degree invariants.

For degree seven and eight, again an invariant monomial $\psi$ must have at least two variable; in degree seven all such $\psi$ are divisible by $X_i^3$ for some $i$, and we can argue as above, and in degree eight, if it is not divisible by $X_i^3$ for any $i$, then it must be 
\[
(X_1X_2X_3X_4)^2,
\]
which is generated by lower degree invariants. 

Finally, in degree nine and higher, every monomial $\psi$ is divisible by $X_i^3$ for some $i$, so as long as it contains at least two variables we are done. On the other hand, in degrees $n$ divislbe by 5, we could have the monomials $X_i^n$, but these are again generated by the invariant quintics $X_i^5$. Therefore we have shown that all invariant polynomials of degree $5n$ where $n\geq 2$ (in the $\mu_5$-action in the $\mb{P}^4$-case) are all by the quintic invariants, as required.

We may now proceed as in \cite{lang}. From the map $\mb{P}^4\rightarrow \mb{P}^{25}$, and the fact that the fixed points of the $\mu_5$-action  on $\mb{P}^4$ are isolated points, a generic hyperplane $H$ will not contain any fixed points. Therefore the threefold 
\[
X\defeq \mb{P}^4/\mu_5\cap H
\]
is smooth by Bertini's theorem, and it is the quotient of the quintic hypersurface in $\mb{P}^4$ defined by $H$, on which the $\mu_5$-action is free.

Finally we show that the quotients $X$ have trivial canonical bundle.
\begin{claim}
We have $\Omega^3_X\cong \mc{O}_X$.
\end{claim}%i.e. that $\Omega^3\cong \mc{O}$. 
\begin{proof}
First observe  that the above construction  works just as well in mixed characteristic; that is, the same construction and arguments as above give varieties over $\mb{Z}_p$ (or perhaps some other local ring-the argument will go through without change), and so for the remainder of the proof we denote by $X$ this integral model, and $X[1/p]$  its generic fiber. We first  show that $\Omega^3$ is trivial in characteristic zero: that is, on the variety $X[1/5]$. For this, we need to find a nowhere vanishing section of $\Omega^3$. Let $Y$  be the $\mu_5$-cover of $X[1/5]$; since we are now in characteristic zero, $Y$ is smooth and it is a quintic hypersurface in $\mb{P}^4$, so certainly $\Omega^3_Y$ is trivial, and let us denote by $\sigma$ a nowhere vanishing section. Now by adjunction we have
\[
\Omega^3_Y\cong \Omega^4_{\mb{P}^4}(5)|_Y\cong\mc{O}_{\mb{P}^4}|_Y,
\]
and taking global sections we have that $\sigma$ is invariant under the $\mu_5$-action, and so descends to $X$; furthermore this descended section is also nowhere vanishing, and hence $\Omega^3_{X[1/p]}$ is trivial, which is what we wanted. 

It remains to show the same for the integral model $X$. On $X$ we have the exact sequence
\[
0\rightarrow \Omega^3_X\xrightarrow{\cdot p}\Omega^3_X\rightarrow \Omega^3_{\overline{X}}\rightarrow 0
\]
where we have denoted by $\overline{X}$ the special fiber of $X$. Hence we have an injection 
\[
H^0(X, \Omega^3_X)/pH^0(X, \Omega^3_X)\hookrightarrow H^0(\overline{X}, \Omega^3_{\overline{X}}). 
\]
Therefore, by taking an appropriate multiple of the section $\sigma$ above, we see that  the line bundle $\Omega_X$ has a section whose vanishing locus has codimension at least 2, and hence is everywhere non-zero. Therefore $\Omega^3_X$ is the trivial line bundle, as required.
\end{proof}
%Furthermore the $\mu_5$-action on 
%\[
%H^3(Y, \Omega^3)\cong H^0(\mb{P}^4, \mc{O})^{\vee} 
%\]
%is trivial, and hence the global section $\sigma$ descends to $X[1/5]$, and is also nowhere vanishing. 

%Now consider the $\mb{Z}_p$-module $M\defeq H^0(X, \Omega^3)$. We have a section 
%\[
%\sigma \in M[1/p]
%\]
%as above, and by scaling $\sigma $ if necessary we may assume that it lives in $M$, and that its image in $M/pM$ is non-zero; we continue to denote this rescaled section by $\sigma$. Therefore  $\sigma$ can vanish only in some proper subvariety of the special fiber of $X$. On the other hand, the special fiber is irreducible, and it is not possible for a section of a line bundle to vanish in a non-empty subset of codimension 2. Therefore we have $\Omega^3\cong \mc{O}$, as claimed.
This concludes the proof of the theorem.
\end{proof}
\begin{rmk}
As a consistency check, we can compute the Hodge numbers of the threefold $X$. In characteristic zero, according to the proof of Theorem \ref{thm:godeauxconstruct} there are 26 $\mu_5$-invariant quintics, and quotienting by the subgroup of $\GL(5)$ stabilizing these quintics (up to scalars), namely the group of diagonal matrices $\mb{G}_m^5$, there are 
\[
21=26-5
\]
possible deformations of $X$. Therefore $h^{21}=21$; note that $h^{11}=1$ (the same as the covering hypersurface $Y$), and so the Euler characteristic is 
\[
\chi(X)=-40=-200/5=\chi(Y)/5,
\]
as expected. 
\end{rmk}

%mathbb
%\printbibliography
%heading=bibintoc,
%title={}
 % Bibliography database file Y.bib
\printbibliography
\end{document}